\documentclass[reqno,11pt]{amsart}
 
\usepackage{amsmath}
\usepackage{amsthm}  
\usepackage{verbatim}
\usepackage{enumerate} 
\usepackage{mathtools}  
\usepackage{amssymb}
\usepackage{mathrsfs}
\usepackage{marginnote}
\usepackage[top=1.5in, bottom=1.5in, left=0.7in, right=0.7in]{geometry}  
\usepackage[colorlinks]{hyperref}
\hypersetup{
	colorlinks,
	citecolor=blue,
	linkcolor=red
}   
\numberwithin{equation}{section}
\usepackage{xypic}
\usepackage{bm}
\usepackage{tikz}
\usetikzlibrary{decorations.pathreplacing}
\usepackage{float}

\newtheorem{theorem}{\textbf{Theorem}}[section]
\newtheorem{theorem*}{\textbf{Theorem}}

\newtheorem{proposition}[theorem]{\textbf{Proposition}}
\newtheorem{lemma}[theorem]{\textbf{Lemma}}

\newtheorem{corollary}[theorem]{\textbf{Corollary}}
\newtheorem{remark}[theorem]{\textbf{Remark}}

\newtheorem{conjecture}[theorem]{\textbf{Conjecture}}
\newtheorem{definition/proposition}[theorem]{\textbf{Definition/Proposition}}
\newtheorem{innercustomgeneric}{\customgenericname}
\providecommand{\customgenericname}{}
\newcommand{\newcustomtheorem}[2]{%
	\newenvironment{#1}[1]
	{%
		\renewcommand\customgenericname{#2}%
		\renewcommand\theinnercustomgeneric{##1}%
		\innercustomgeneric
	}
	{\endinnercustomgeneric}
}
\newcustomtheorem{customthm}{Theorem}
\newcustomtheorem{customprop}{Proposition}
\newcustomtheorem{customcoro}{Corollary}

\def\N{{\mathbb N}}
\def\R{\mathbb{R}}
\def\Z{{\mathbb Z}}

\def\C{{\mathbb C}}

\def\Q{{\mathbb Q}}

\newcommand{\CP}{\mathbb{C}\mathbb{P}}
\newcommand{\RP}{\mathbb{R}\mathbb{P}}

\def\cA{{\mathcal A}}

\def\cC{{\mathcal C}}

\def\cH{{\mathcal H}}

\def\cM{{\mathcal M}}

\def\cO{{\mathcal O}}

\def\rd{{\rm d}}

\def\la{\langle\,}
\def\ra{\,\rangle}

\DeclareMathOperator{\Ima}{im}
\DeclareMathOperator{\ind}{ind}

\DeclareMathOperator{\Id}{id}

\DeclareMathOperator{\coker}{coker}

\title{$(\RP^{2n-1},\xi_{std})$ is not exactly fillable for $n\ne 2^k$}
\author{Zhengyi Zhou}

\begin{document}
	\maketitle
\begin{abstract}
We prove that $(\RP^{2n-1},\xi_{std})$ is not exactly fillable for any $n\ne 2^k$ and there exist strongly fillable but not exactly fillable contact manifolds for all dimension $\ge 5$.
\end{abstract}
\section{Introduction}
One fundamental principle in contact topology is the dichotomy between overtwisted and tight contact structures discovered by Eliashberg \cite{eliashberg1989classification} in dimension $3$. The dichotomy was generalized to all higher dimensions recently  by Borman, Eliashberg, and Murphy \cite{borman2015existence}. The $h$-principle for overtwisted contact structures implies that they are governed by their underlying formal data. On the other hand, the more mysterious tight contact structures can be roughly categorized into the following classes based on their fillability.
$$\{\text{Weinstein fillable}\}\subseteq \{\text{Exactly fillable}\} \subseteq \{\text{Strongly fillable}\} \subseteq \{\text{Weakly fillable}\} \subseteq \{\text{Tight}\}.$$

It is an interesting question to study differences between these classes. In dimension three, those inclusions were shown to be proper by Bowden \cite{bowden2012exactly}, Ghiggini \cite{ghiggini2005strongly}, Eliashberg \cite{eliashberg1996unique}, Etnyre and Honda \cite{etnyre2002tight} respectively. In higher dimensions, the first, third and fourth inclusions were shown to be proper by Bowden, Crowley, and Stipsicz \cite{bowden2014topology}, Bowden, Gironella, and Moreno \cite{bowden2019bourgeois}, Massot, Niederkr{\"u}ger, and Wendl \cite{massot2013weak}, who also first showed the properness of the third inclusion in dimension five. See also \cite{zhou2019symplectic} for exactly fillable, almost Weinstein fillable, but not Weinstein fillable examples. The situation in dimension three differs from higher dimensions in the sense that we have gauge theoretic tools as well as better holomorphic curve theories, but also face more topological constraints. In higher dimensions, we have fewer tools but more flexibility in constructions. The first, third and fourth inclusions can be studied from more structured perspectives. The challenges in those cases are finding examples and executing the machineries. On the other hand,  it seems that we are poorly equipped to study the second inclusion in higher dimensions.  Fortunately, we have simple potential examples, i.e.\ real projective spaces with the standard contact structure induced from the double cover. It was conjectured by Eliashberg \cite[\S 1.9]{question} that they are not exactly fillable whenever the dimension is greater than $3$. In this paper, we verify this conjecture for most cases.

\begin{theorem}\label{thm:main}
	For $n\ne 2^k$, $(\RP^{2n-1},\xi_{std})$ is not exactly fillable.
\end{theorem}
Note that $(\RP^{2n-1},\xi_{std})$ admits a strong filling $\cO(-2)$, i.e.\ the degree $-2$ line bundle over $\CP^{n-1}$, which is not exact since the zero section $\CP^{n-1}$ is a symplectic submanifold. The condition $n\ge 3$ is necessary, since $(\RP^3,\xi_{std})$ is exactly (Weinstein) fillable by $T^*S^2$.  Moreover, it is known from \cite[\S 6.2]{eliashberg2006geometry} that $(\RP^{2n-1},\xi_{std})$ admits a Weinstein filling iff $n\le 2$. 

The proof in this paper consists of
a symplectic and a topological argument, see below. The symplectic part of the proof in this paper only requires $n\ge 3$, see Remark \ref{rmk:n=2} for how our proof sees the exception when $n=2$. The $n\ne 2^k$ condition is only used in the  topological argument in Proposition \ref{prop:Chern}.  Our proof also shows that $(\RP^{2n-1},\xi_{std}), n\ne 2^k$ admits no symplectically aspherical filling and no Calabi-Yau filling (i.e.\ strong fillings $W$ such that $c_1(W)=0\in H^2(W;\Q)$), see Remark \ref{rmk:other}.
\begin{remark}
	The $n=3$ case was also announced by Ghiggini and Niederkr{\"u}ger using a different method.
\end{remark}

More specifically, our strategy of proof is the following.

Symplectic part: $(\RP^{2n-1},\xi_{std})$ has a very nice Reeb dynamics, moreover, its double cover is the standard contact sphere. In the latter, a short Reeb orbit will bound a rigid holomorphic curve with a point constraint. For a generic point or a generic almost complex structure,  the curve lives completely in the symplectization. In particular, one can push the curve down to the hypothetical exact filling of $\RP^{2n-1}$. Such curve kills a unit hence the whole symplectic cohomology for any exact filling of $(S^{2n-1},\xi_{std})$. Then we argue that this is also the case for $(\RP^{2n-1},\xi_{std})$ for $n\ge 3$ by the curve  pushed down. Moreover, we know exactly at which stage a unit is killed, this allows us to use the filtered positive symplectic cohomology to estimate the rank of the cohomology of the filling, and prove that it is at most two.  

Topological part: We argue that it is impossible to have an almost complex filling with such small free part. For this we need to find some nontrivial restriction $H^i(W)\to H^i(\RP^{2n-1})$. We can study this map by computing Chern classes of $\xi_{std}$ since $c_i(W)|_{\RP^{2n-1}}=c_i(\xi_{std})$. The $n\ne 2^k$ condition is used to ensure the total Chern class of $\xi_{std}$ is not trivial.  

Note that $(\RP^{2n-1},\xi_{std})$ can be viewed as the link of the quotient singularity $\C^n/\Z_2$. Our method is adaptable to other quotient singularities as well. We use $(S^{2n-1}/\Z_k,\xi_{std})$ to denote the link of $\C^n/\Z_k$, where $\Z_k$ acts on $\C^n$ by multiplication by $e^{\frac{2\pi i}{k}}$. Then we prove the following.
\begin{theorem}\label{thm:gen}
Let $p$ be an odd prime, assume $n$ has the $p$-adic representation $n=\sum_{s=0}^k a_s p^s$. Then $(S^{2n-1}/\Z_p, \xi_{std})$ has no exact filling if $\sum_{s=0}^k a_s>3p-3$.
\end{theorem}
Note that the $n=2$ case is again exactly fillable, see \cite{eliashberg1993legendrian}.  The threshold  is by no means sharp, and the $p$-adic information is very likely unnecessary. We will speculate in the proof that the symplectic part works for  $n\ge p+1$. While the $p$-adic information of $n$ is used in the topological argument to guarantee that the total Chern class of $(S^{2n-1}/\Z_p, \xi_{std})$ has enough nontrivial terms. Note that $n=p$ is the threshold for the quotient singularity to be canonical, or admit a crepant resolution, and when $n>p$ the singularity becomes terminal. As explained in \cite{mclean2016reeb}, being terminal is equivalent to having positive minimal SFT degree for some contact form, which is closely related to the concept of asymptotically dynamical convexity \cite{lazarev2016contact}\footnote{The main difference between \cite{mclean2016reeb} and \cite{lazarev2016contact} lies in the treatment of non-contractible orbits, which play important roles in this paper. More precisely, \cite{mclean2016reeb} assigned a rational SFT degree to every Reeb orbits in the case of $c_1(\xi)=0\in H^2(Y,\Q)$ and $H^1(Y,\Q)=0$ and being terminal is equivalent to that there exists a contact form such that all Reeb orbits have positive SFT degree. On the other hand, \cite{lazarev2016contact} only considered contractible Reeb orbits in the case of  $c_1(\xi)=0$ and asymptotically dynamical convexity is roughly admitting a contact form such that all contractible Reeb orbits have positive (integer) SFT degree.  In the case of $(\RP^{2n-1},\xi_{std}),n\ge 3$, since those non-contractible simple Reeb orbits play an important role, the relevant concept is the former one.}. In particular, the result here bears certain similarity with \cite{zhou2019symplectic}. However, we do not assume the exact filling to have any topological properties (e.g.\ vanishing first Chern class and $\pi_1$-injectivity) as in \cite{zhou2019symplectic}.  Our approach can be adapted to study more general quotient singularity $\C^n/G$. The relation between exact fillability and its algebro-geometric properties is an interesting question, we wish to study it in the future.

Combining with the $\Z_2$ and $\Z_3$ quotient singularities, we will show that the second inclusion is proper for all dimension $\ge 5$, hence complete the question of proper inclusions for all dimension $>1$.
\begin{theorem}\label{thm:strong}
	For every $n\ge 3$, then there exists a $2n-1$ dimensional contact manifold which is strongly fillable but not exactly fillable. 
\end{theorem}

\subsection*{Acknowledgement} The author is supported by the National Science Foundation Grant No. DMS1638352. It is a great pleasure to acknowledge the Institute for Advanced Study for its warm hospitality. The author is in debt to Paolo Ghiggini for very helpful comments and pointing out a gap in an earlier draft. The author is deeply grateful to anonymous referees for many helpful comments and suggestions. This paper is dedicated to the memory of Chenxue.

\section{Proof of Theorem \ref{thm:main}}\label{s2}
In the following, the coefficient is $\Z$ unless otherwise specified. The contradiction leading to the proof of Theorem \ref{thm:main} is the following.
\begin{proposition}\label{prop:key}
	For $n\ge 3$, if $(\RP^{2n-1},\xi_{std})$ has an exact filling $W$, then the following holds.
	\begin{enumerate}
		\item If $n$ is odd, then the total cohomology $H^*(W;\R)=\R$, supported in degree $0$.
		\item If $n$ is even, then $H^*(W;\R) =\R$ or $\R \oplus\R$, and in the latter case, the cohomology is precisely supported in degree $0$ and $n$. 
	\end{enumerate}
\end{proposition}
We will first prove Theorem \ref{thm:main} assuming Proposition \ref{prop:key}. First of all, we observe the following fact.
\begin{proposition}\label{prop:Chern}
Let $W$ be a strong filling of $(\RP^{2n-1},\xi_{std})$ for $n=2^kp$ for odd $p\ge3$, then $H^{2^{k+1}}(W) \to H^{2^{k+1}}(\RP^{2n-1})=\Z_2$ and  $H^{2n-2^{k+1}}(W) \to H^{2n-2^{k+1}}(\RP^{2n-1})=\Z_2$, induced by the inclusion $\RP^{2n-1}\hookrightarrow W$,  are both surjective.
\end{proposition}
\begin{proof}
	Since $c_i(W)|_{\RP^{2n-1}} = c_i(\xi_{std})$, we can prove the claim if both $c_{2^k}(\xi_{std})$ and $c_{n-2^k}(\xi_{std})$ are nonzero. In the following, we will compute the total Chern class of $\xi_{std}$ from the standard filling $\cO(-2)$. The total Chern class of the total space $\cO(-2)$ can be computed from the total Chern class of $T\cO(-2)|_{\CP^{n-1}}$, where the bundle splits into $T\CP^{n-1}\oplus \cO(-2)$. Since the total Chern class of $T\CP^{n-1}$ is $(1+u)^n$ and the total Chern class of the bundle $\cO(-2)$ is $1-2u$ \cite[Theorem 14.4]{milnor2016characteristic}, where $u$ is the generator of $H^2(\CP^{n-1})=H^2(\cO(-2))$. Hence the total Chern class of the total space $\cO(-2)$ is $(1+u)^n(1-2u)$.  
	
	Using the fact that $(\sum a_i)^2 = \sum a_i^2 \mod 2$, we have
	$$(1+u)^n(1-2u)=(1+u)^{p2^k}=(1+pu+\ldots+pu^{p-1}+u^p)^{2^k}=1+p^{2^k}u^{2^k}+ \ldots + p^{2^k}u^{(p-1)2^k} + u^{n}  \mod 2.$$
	Therefore we have both $c_{2^k}(\cO(-2))$ and $c_{n-2^k}(\cO(-2))$ are not zero in $\Z_2$. By the Gysin exact sequence, the restriction map $\Z = H^{2i}(\cO(-2))\to H^{2i}(\RP^{2n-1})=\Z_2$ is the mod $2$ map for $1\le i \le n-1$. Then the claim follows.
\end{proof}
\begin{remark}
	When $n=2^k$, the total Chern class of $\xi_{std}$ is trivial.
\end{remark}

\begin{proof}[Proof of Theorem \ref{thm:main}]
	Assume $(\RP^{2n-1},\xi_{std})$ has an exact filling $W$ and $n=2^kp$ for odd $p\ge 3$. Then by Proposition \ref{prop:key}, $H^{2^{k+1}-1}(W)$, $H^{2^{k+1}}(W)$, $H^{2^{k+1}+1}(W)$, $H^{2n-1-2^{k+1}}(W)$, $H^{2n-2^{k+1}}(W)$, $H^{2n+1-2^{k+1}}(W)$ are all torsions. By looking at the long exact sequence of $(W,\RP^{2n-1})$, we have the following,
	$$0\to H^{2^{k+1}}(W,\RP^{2n-1})\to H^{2^{k+1}}(W) \to \Z_2 \to H^{2^{k+1}+1}(W,\RP^{2n-1})\to H^{2^{k+1}+1}(W)\to 0,$$
	$$0 \to H^{2n-2^{k+1}}(W,\RP^{2n-1})\to H^{2n-2^{k+1}}(W) \to \Z_2 \to H^{2n+1-2^{k+1}}(W,\RP^{2n-1})\to H^{2n+1-2^{k+1}}(W) \to 0,$$
	which also implies that all groups above are torsions. Then by Lefschetz duality and universal coefficient theorem, we have 
	$$H^{2^{k+1}}(W,\RP^{2n-1}) \simeq H_{2n-2^{k+1}}(W) \simeq H^{2n+1-2^{k+1}}(W),$$
	$$H^{2^{k+1}+1}(W,\RP^{2n-1}) \simeq H_{2n-1-2^{k+1}}(W) \simeq H^{2n-2^{k+1}}(W),$$
	$$H^{2n-2^{k+1}}(W,\RP^{2n-1}) \simeq H_{2^{k+1}}(W) \simeq H^{2^{k+1}+1}(W),$$
	$$H^{2n+1-2^{k+1}}(W,\RP^{2n-1}) \simeq H_{2^{k+1}-1}(W) \simeq H^{2^{k+1}}(W).$$
	Therefore the two long exact sequences become
	$$0\to H^{2n+1-2^{k+1}}(W)\to H^{2^{k+1}}(W) \to \Z_2 \to H^{2n-2^{k+1}}(W)\to H^{2^{k+1}+1}(W)\to 0,$$
	$$0 \to H^{2^{k+1}+1}(W)\to H^{2n-2^{k+1}}(W) \to \Z_2 \to H^{2^{k+1}}(W)\to H^{2n+1-2^{k+1}}(W) \to 0.$$
	By Proposition \ref{prop:Chern}, $H^{2^{k+1}}(W) \to \Z_2$ and $H^{2n-2^{k+1}}(W) \to \Z_2 $ above are surjective. Then the long exact sequences above imply that $H^{2n-2^{k+1}}(W)\simeq H^{2^{k+1}+1}(W)$. Since all those groups are finite groups, we have $|H^{2n-2^{k+1}}(W)|=|H^{2^{k+1}+1}(W)|$ and also $|H^{2n-2^{k+1}}(W)|=2|H^{2^{k+1}+1}(W)|$, which is a contradiction.
\end{proof}
The rest of the section is devoted to the proof of Proposition \ref{prop:key}.

\subsection{Setup of symplectic cohomology}
Note that $(\RP^{2n-1},\xi_{std})$ is equipped with the Boothby-Wang contact form $\alpha_{std}$ such that the Reeb vector gives the Hopf fibration, and the periods of Reeb orbits are given by $\N^+$. We can choose a $C^2$-small perfect Morse function $f$ on $\CP^{n-1}$. Let $q_0,\ldots,q_{n-1}$ denote the critical points of $f$ ordered by their critical values. Let $\pi:\RP^{2n-1}\to \CP^{n-1}$ denote the projection, then the $r=1+\pi^*f$ hyperspace in the symplectization $\RP^{2n-1}\times (\R_+)_r$ gives a perturbed contact form $\alpha_f$, such that the Reeb vector field for $\alpha_f$ is 
\begin{equation}\label{eqn:Reeb}
R_f=\frac{1}{1+\pi^*f} R_{std}+Z, \text{ with } Z \in \xi_{std}, \iota_Z d\alpha_{std}|_{\xi_{std}}=-\frac{1}{(1+\pi^*f)^2}\pi^*df|_{\xi_{std}}.
\end{equation}
In other words, $Z$ is the horizontal lift of the Hamiltonian vector of $\frac{1}{1+\pi^*f}$ on $(\CP^{n-1},\omega_{FS})$ using $\alpha_{std}$ as a connection on $\RP^{2n-1}$. We may fix a small $\epsilon>0$ with $2n\epsilon<1$, such that $f(q_i)=i\epsilon$. Since $f$ is $C^2$-small, we may assume that the Hamiltonian vector of $\frac{1}{1+\pi^*f}$, i.e.\ $\pi_*Z$, has no non-constant orbit of period smaller than $3$. As a consequence of \eqref{eqn:Reeb}, we have the following.
\begin{enumerate}
	\item There is a simple Reeb orbit $\gamma_i$, such that $\pi(\gamma_i)=q_i$ and the period of $\gamma_i$ is $1+i\epsilon=1+f(q_i)$. 
	\item All Reeb orbits of period smaller than $2(1+n\epsilon)$ are non-degenerate and is either $\gamma_i$ or its double cover $\gamma^2_i$.
\end{enumerate}
Let $p:S^{2n-1}\to \RP^{2n-1}$ denote the double cover. Then $(S^{2n-1},p^*\alpha_{f})$ is roughly an ellipsoid, in particular, all Reeb orbits of period smaller than $2(1+n\epsilon)$ are the lifts of $\gamma_i^2$. Among them, the lift of $\gamma_0^2$ has the minimal period and minimal Conley-Zehnder index $n+1$.

In the following, we will fix a specific choice of $f$. Let $\epsilon$ be a small positive number, we choose $f$ to be 
$$\epsilon \left(\displaystyle\sum_{i=0}^{n-1} \frac{i}{1+i\epsilon}|z_i|^2 \right)\left/ \left(\displaystyle\sum_{i=0}^{n-1} \frac{|z_i|^2}{1+i\epsilon}\right)\right..$$
Then the function $f$ satisfies all conditions above. Moreover, $(S^{2n-1},p^*\alpha_f)$ is indeed the ellipsoid given by $\sum_{i=0}^{n-1}\frac{|z_i|^2}{2(1+i\epsilon)} = 1$ in $(\C^n,\frac{i}{2\pi}\sum_{i=0}^{n-1}\rd z_i\wedge \rd \overline{z}_i)$. Here when we identify the symplectic manifold $(\RP^{2n-1}\times \R_+, d(r\alpha_{std}))$ with $(\{\C^n-\{0\}\}/\Z_2,\frac{i}{2\pi}\sum_{i=0}^{n-1} dz_i\wedge d\overline{z}_i)$, we have $r=\frac{1}{2}\sum_{i=0}^{n-1} |z_i|^2$. An ellipsoid $\sum_{i=0}^{n-1} \frac{|z_i|^2}{a_i}=1$ is called non-degenerate iff $a_i/a_j\not \in \Q$ for any $i\ne j$. If the ellipsoid is non-degenerate, the induced contact form is non-degenerate and all Reeb orbits are circles in each coordinate plane. In our case, if we pick $\epsilon\notin \Q$, $(S^{2n-1},p^*\alpha_f)$ is a non-degenerate ellipsoid.

Let $(W,\lambda)$ be an exact filling of $(\RP^{2n-1},\alpha_f)$, i.e.\ the Liouville form $\lambda$ restricted to $\RP^{2n-1}$ is $\alpha_f$. Let $\widehat{W}:=W\cup \partial W \times (1,\infty)$ denote the completion. We will set up our moduli spaces for symplectic cohomology following \cite{zhou2019symplectic} combined with the autonomous setting in \cite{bourgeois2009symplectic}. Instead of using a single Hamiltonian as in \cite{zhou2019symplectic}, we will phrase symplectic cohomology as a direct limit over Hamiltonians with finite slopes like the classical construction in \cite{seidel2006biased}. We refer to them as well as references therein for details of symplectic cohomology. In particular, we will use Hamiltonians and almost complex structures satisfies the following.
\begin{enumerate}
	\item $H=0$ on $W$ and $H=h(r)$ on $\partial W \times (1,\infty)$ with $h'(r)=a$ for $r\gg 0$ and $h''(r)>0$ unless $h=0$ or $h'(r)=a$. We will be only interested in the case $a \in (0,2(1+n\epsilon))$ and is not the period of a Reeb orbit. The class of Hamiltonians with slope $a$ is denoted as $\cH_a$. Since $h'(r)=a$ iff $r\ge 1+w$ for some $w>0$, we will call the minimum $w$ with such property the width of the Hamiltonian $H$.
	\item The almost complex structure $J_t$ is independent of $t$ on $W\cup \partial W \times (1,r_0]$, where $h'(r_0)=1+(n-1)\epsilon$, i.e.\ $r_0$ is the last level containing a simple Hamiltonian orbit. $J_t$ is compatible with symplectic structure and is cylindrical convex near every $r$ such that $h'(r)$ is the period of a Reeb orbit, i.e.\ $J_t\xi=\xi$ and $J_t(r\partial r) = R_{\alpha_f}$. This guarantees the integrated maximum principle \cite{abouzaid2010open} can be applied to obtain compactness of moduli spaces, see \cite[Lemma 2.5]{zhou2019symplectic} for details. 
\end{enumerate}
We also choose a Morse function $g$ on $W$, such that $\partial_r g>0$ on $\partial W$ and $g$ has a unique minimum. The extra requirement on $g$ will be specified later. We use $\overline{\gamma}$ to denote the $S^1$ family of Hamiltonian orbits corresponding to $\gamma$. Then we pick two different generic points $\hat{\gamma}$ and $\check{\gamma}$ on $\Ima \overline{\gamma}$, this is equivalent to choosing a Morse function $g_{\overline{\gamma}}$ with one maximum and one minimum on $\Ima \overline{\gamma}$ in \cite[\S 3]{bourgeois2009symplectic}. By \cite[Lemma 3.4]{bourgeois2009symplectic}, the Morse function $g_{\overline{\gamma}}$ can be used to perturb the Hamiltonian $H$ to get two non-degenerate orbits from $\overline{\gamma}$, which are often denoted by $\hat{\gamma}$ and $\check{\gamma}$ in literatures with $\mu_{CZ}(\hat{\gamma})=\mu_{CZ}(\gamma)+1$ and $\mu_{CZ}(\check{\gamma})=\mu_{CZ}(\gamma)$.   

Then we have a Floer cochain complex $C(H)$, which is a free $\R$-module generated by critical points of $g$ with Morse index as grading, and two generators $\hat{\gamma},\check{\gamma}$ for each Reeb orbit $\gamma$ of period smaller than $a$, with $\Z_2$ gradings $n-\mu_{CZ}(\gamma)-1$ and $n-\mu_{CZ}(\gamma)$.   The differential is defined by counting rigid cascades \cite[(39),(40)]{bourgeois2009symplectic}. Moreover, we have a subcomplex $(C_0(H),d_0)$, which is the Morse cochain complex of $g$ and a quotient complex $(C_+(H),d_+):=C(H)/C_0(H)$ generated by the generators from Reeb orbits. The differential on $C(H)$ also  has a component $d_{+,0}:C_+(H)\to C_0(H)$, which induces the connecting map $H^*(C_+(H))\to H^{*+1}(C_0(H))$ in the tautological long exact sequence. We can achieve transversality using our almost complex structure, because on $r\le r_0$ all orbits are simple \cite[Proposition 3.5]{bourgeois2009symplectic}. The differential can be described in a pictorial way as follows.
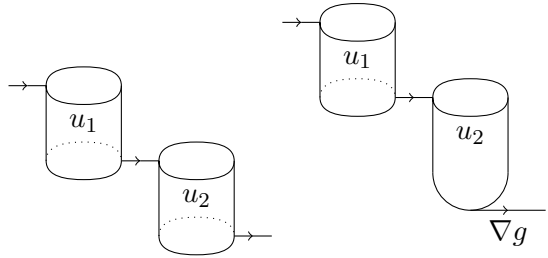
\begin{figure}[H]
	\begin{tikzpicture}
	\draw (0,0) to [out=90, in = 180] (0.5, 0.25) to [out=0, in=90] (1,0) to [out=270, in=0] (0.5,-0.25)
	to [out = 180, in=270] (0,0) to (0,-1);
	\draw[dotted] (0,-1) to  [out=90, in = 180] (0.5, -0.75) to [out=0, in=90] (1,-1);
	\draw (1,-1) to [out=270, in=0] (0.5,-1.25) to [out = 180, in=270] (0,-1);
	\draw (1,0) to (1,-1);
	\draw[->] (1,-1) to (1.25,-1);
	\draw (1.25,-1) to (1.5,-1);
	\draw (1.5,-1) to [out=90, in = 180] (2, -0.75) to [out=0, in=90] (2.5,-1) to [out=270, in=0] (2,-1.25)
	to [out = 180, in=270] (1.5,-1) to (1.5,-2);
	\draw[dotted] (1.5,-2) to  [out=90, in = 180] (2, -1.75) to [out=0, in=90] (2.5,-2);
	\draw (2.5,-2) to [out=270, in=0] (2,-2.25) to [out = 180, in=270] (1.5,-2);
	\draw (2.5,-1) to (2.5,-2);
	\draw[->] (2.5,-2) to (2.75,-2);
	\draw (2.75,-2) to (3,-2);
	\draw[->] (-0.5,0) to (-0.25,0);
	\draw (-0.25,0) to (0,0);
	\node at (0.5,-0.5) {$u_1$};
	\node at (2, -1.5) {$u_2$};
	\end{tikzpicture}
	\begin{tikzpicture}
	\draw (0,0) to [out=90, in = 180] (0.5, 0.25) to [out=0, in=90] (1,0) to [out=270, in=0] (0.5,-0.25)
	to [out = 180, in=270] (0,0) to (0,-1);
	\draw[dotted] (0,-1) to  [out=90, in = 180] (0.5, -0.75) to [out=0, in=90] (1,-1);
	\draw (1,-1) to [out=270, in=0] (0.5,-1.25) to [out = 180, in=270] (0,-1);
	\draw (1,0) to (1,-1);
	\draw[->] (1,-1) to (1.25,-1);
	\draw (1.25,-1) to (1.5,-1);
	\draw (1.5,-1) to [out=90, in = 180] (2, -0.75) to [out=0, in=90] (2.5,-1) to [out=270, in=0] (2,-1.25)
	to [out = 180, in=270] (1.5,-1) to (1.5,-2);
	\draw (2.5,-2) to [out=270, in=0] (2,-2.5) to [out = 180, in=270] (1.5,-2);
	\draw (2.5,-1) to (2.5,-2);
	\draw[->] (-0.5,0) to (-0.25,0);
	\draw (-0.25,0) to (0,0);
	\draw[->] (2,-2.5) to (2.5,-2.5);
	\draw (2.5,-2.5) to (3,-2.5);
	\node at (2.5, -2.8) {$\nabla g$};  
	\node at (0.5,-0.5) {$u_1$};
	\node at (2, -1.5) {$u_2$};
	\end{tikzpicture}
	\caption{$d_+$ and $d_{+,0}$ from $2$ level cascades}
\end{figure}
\begin{enumerate}
	\item Each unlabeled horizontal arrow is a negative gradient flow of $g_{\overline{\gamma}}$ in $\Ima \overline{\gamma}$, i.e.\ flow toward $\check{\gamma}$.
	\item $u$ is a solution to the Floer equation $\partial_s u+J_t(\partial_tu-X_H)=0$ modulo $\R$ translation.
	\item Every intersection point of line with surface satisfies the obvious matching condition.
\end{enumerate}
More formally, the differential is defined by counting the following compactified moduli spaces.
\begin{enumerate}
	\item For $p,q\in Crit(g)$, 
	$$\cM_{p,q}:=\overline{\{ \gamma:\R_s \to W|\gamma'+\nabla g =0, \lim_{s\to \infty} \gamma=p,\lim_{s\to -\infty} \gamma=q \}/\R}.$$
	\item For $\gamma_+,\gamma_- \in \{\check{\gamma},\hat{\gamma}| \forall S^1 \text{ family of orbits } \overline{\gamma}\}$, a $k$-cascade from $\gamma_+$ to $\gamma_-$ is a tuple $(u_1,l_1,\ldots,l_{k-1},u_k)$, such that
	\begin{enumerate}
		\item $l_i$ are positive real numbers.
		\item nontrivial $u_i\in \{u:\R_s\times S^1_t\to \widehat{W}| \partial_su+J_t(\partial_tu-X_H)=0, \lim_{s\to \infty} u \in \overline{\gamma}_{i-1}, \lim_{s\to -\infty} u \in \overline{\gamma}_{i}\}/\R$ such that $\gamma_+\in \overline{\gamma}_0$ and $\gamma_-\in \overline{\gamma}_{k}$, where the $\R$ action is the translation on $s$.
		\item $\phi_{-\nabla g_{\overline{\gamma}_i}}^{l_i}(\lim_{s\to -\infty}u_i(s,0))=\lim_{s\to \infty}u_{i+1}(s,0)$ for $1\le i\le k-1$, $\gamma_+=\lim_{t\to \infty}\phi^{-t}_{-\nabla g_{\overline{\gamma}_0}}(\lim_{s\to \infty} u_1(s,0))$, and $\gamma_-=\lim_{t\to \infty}\phi^{t}_{-\nabla g_{\overline{\gamma}_k}}(\lim_{s\to -\infty} u_k(s,0))$, where $\phi^t_{-\nabla g_{\overline{\gamma}}}$ is the time $t$ flow of $-\nabla g_{\overline{\gamma}}$ on $\Ima \overline{\gamma}$.
	\end{enumerate}
	Then we define $\cM_{\gamma_+,\gamma_-}$ to be the compactification of the space of all cascades from $\gamma_+$ to $\gamma_-$. The compactification involves the usual Hamiltonian-Floer breaking of $u_i$ as well as degeneration corresponding to $l_i=0,\infty$.  The $l_i=0$ degeneration is equivalent to a Hamiltonian-Floer breaking $\lim_{s\to -\infty}u_i=\lim_{s\to \infty}u_{i+1}$. In particular, they can be glued or paired, hence do not contribute (algebraically) to the boundary of $\cM_{\gamma_+,\gamma_-}$. The $l_i=\infty$ degeneration is equivalent to a Morse breaking for $g_{\overline{\gamma}_i}$, which will contribute to the boundary of $\cM_{\gamma_+,\gamma_-}$.

	\item For $\gamma_+ \in  \{\check{\gamma},\hat{\gamma}| \forall S^1 \text{ family of orbits } \overline{\gamma}\}$ and $q\in Crit(g)$, a $k$-cascades from $\gamma_+$ to $q$ is a tuple  $(l_0,u_1,l_1,\ldots,u_k,l_{k})$ as before, except
	\begin{enumerate}
		\item $u_k\in \{u:\C \to \widehat{W}|\partial_s u+J_t(\partial_tu-X_H)=0,\lim_{s\to \infty} u \in \overline{\gamma}_{k-1}, u(0)\in W^{\circ}\}/\R$, where we use the identification $\R\times S^1 \to \C^*, (s,t)\mapsto e^{2\pi(s+it)}$. $W^{\circ}$ is the interior of $W$, where the Floer equation is $\overline{\partial}_J u=0$, hence the removal of singularity implies that $u(0)$ is a well-defined notation.
		\item $q=\lim_{t\to \infty}\phi^t_{\nabla g}(u_k(0))$.
	\end{enumerate}
    Then $\cM_{\gamma_+,q}$ is defined to be the compactification of the space of all cascades from $\gamma_+$ to $q$. 
\end{enumerate}
\begin{remark}
	It is important to note that in $\cM_{x,y}$, $x$ is the asymptotic orbit at the positive and $y$ is the asymptotic orbit at the negative end, which is opposite to the notation in \cite{zhou2019symplectic}.
\end{remark}
All of the moduli spaces above can be equipped with coherent orientations. In view of the notation above, the differentials are defined as follows.
\begin{eqnarray*}
d_0 (p) &= &\sum_{q,\dim \cM_{p,q}=0} \# \cM_{p,q} q,\\
d_+ (\gamma_+) &= &\sum_{\gamma_-,\dim \cM_{\gamma_+,\gamma_-}=0} \# \cM_{\gamma_+,\gamma_-} \gamma_-,\\
d_{+,0} (\gamma_+) &= &\sum_{q,\dim \cM_{\gamma_+,q}=0} \# \cM_{\gamma_+,q} q,
\end{eqnarray*}
where $\# \cM$ denotes the signed count of the zero-dimensional compact moduli space $\cM$. Our symplectic action follows the cohomological convention
$$\cA(\gamma) = -\int_{\gamma} \lambda + \int_{\gamma} H,$$
where $\lambda$ is a Liouville form such that $\lambda|_{\partial W}=\alpha_f$. Our convention for $X_H$ is $\rd \lambda (\cdot, X_H)=\rd H$. Every non-constant Hamiltonian orbit is contained in a level set $\partial W \times \{r\}$, such that $h'(r)$ is the period of a Reeb orbit. The symplectic action of this Hamiltonian orbit is given by $-rh'(r)+h(r)$. Since for any non-trivial solution $u$ solving $\partial_s u+J_t(\partial_tu-X_H)=0$, we have $\cA(u(\infty))<\cA(u(-\infty))$, hence if $u(\infty)\in \overline{\gamma}$ and $u(-\infty)\in \overline{\beta}$, then the period of $\gamma$ is larger than the period of $\beta$. 

\begin{remark}\label{rmk:SH}
	A few remarks on different models of symplectic cohomology are in order.
	\begin{enumerate}
		\item\label{1} The most classical construction is using linear time-dependent Hamiltonians, which is also $C^2$ small Morse on $W$, the total symplectic cohomology is the direct limit of the Hamiltonian-Floer cohomology when the slope converges to $\infty$, c.f. \cite{seidel2006biased}.
		\item\label{2} One can also use Hamiltonians that is autonomous. In particular, $H=h(r)$ on the cylindrical end $\partial W \times (1,\infty)$ with $h''>0$ can be used iff the contact form is non-degenerate. Then cascades model is needed to deal with the $S^1$-family of Hamiltonian orbits. This is the construction in \cite{bourgeois2009symplectic}.
		\item\label{3} One can use time-dependent Hamiltonians that are zero on $W$ and are small perturbations to the Hamiltonian in \eqref{2} on the cylindrical end. Then it is very close to a Morse-Bott situation and we need to introduce an auxiliary Morse function $g$ on $W$ such that $\partial_r g>0$ on $\partial W$. This is the approach taken in \cite{zhou2019symplectic}. The vanishing of the Hamiltonian on $W$ makes it easier to apply neck-stretching in this setup.
		\item\label{4} When the contact form is only Morse-Bott non-degenerate in the sense of \cite{bourgeois2003morse}, if we take Hamiltonians in the form of $h(r)$. The Hamiltonian orbits will come in more general family than $S^1$-family. Then we can pick an auxiliary Morse function on the space parameterizing the family and apply a cascade construction. Moreover, one can also choose $H$ to be zero on $W$ and pick another auxiliary Morse function on $W$. This is the approach taken in \cite[\S 4.1]{diogo2019symplectic}\footnote{This complex (called presplit Floer complex in \cite{diogo2019symplectic}) does not require the monotonicity condition in \cite{diogo2019symplectic}.}.
	\end{enumerate}
The approach taken in this paper is a mixture of \eqref{2} and \eqref{3} for moduli spaces setup but with finite slope Hamiltonians as in \eqref{1}, and can also be viewed as a special case of \eqref{4} with a Hamiltonian of finite slope. The compactness of relevant moduli spaces follows from \cite[Proposition 2.6.]{zhou2019symplectic} and \cite[\S 4.2]{bourgeois2009symplectic}. The transversality essentially follows from somewhere injectivity of Floer cylinders, see \cite[Proposition 2.8.]{zhou2018quotient} and \cite[\S 4.1]{bourgeois2009symplectic}. One way to relate the cascades construction with the classical construction is through a gluing analysis for degeneration \cite{bourgeois2009symplectic}, another approach is via cascades continuation/Viterbo transfer maps used in \cite[\S 5]{diogo2019symplectic}. 
\end{remark}

\begin{remark}\label{rmk:CZ}
	To define a global Conley-Zehnder index, we need to choose a trivialization of the determinant line bundle $\det_{\C} \xi_{std}$. In our case, $c_1(\xi_{std})$ is not always $0$  in $H^2(\RP^{2n-1})$ for any $n$, therefore, we may not be able to trivialize $\det_{\C} \xi_{std}$ globally. In this case, we can assign a Conley-Zehnder index for each orbit $\gamma$ after fixing a trivialization of $\gamma^*\det_{\C} \xi_{std}$. The parity of the Conley-Zehnder index does not depend on the trivialization. One way to get a natural trivialization of $\det_{\C} \xi_{std}$ is by choosing a bounding disk $u$ of $\gamma$ either in $\RP^{2n-1}$ or some symplectic filling $W$ of $\RP^{2n-1}$. Since $u^*\det_{\C} \xi_{std}/u^*\det_{\C}TW$ is uniquely trivialized, it induces a trivialization of $\gamma^*\det_{\C} \xi_{std}$. When using Conley-Zehnder index from bounding disks to compute virtual dimensions of moduli spaces, it is crucial to check the bounding disks are compatible via gluing. The Conley-Zehnder index of Hamiltonian orbits has the same property. By computing the indexes in the standard filling $\cO(-2)$, we have that all check generators $\check{\gamma}$ have odd grading and all hat generators $\hat{\gamma}$ have even grading.
	
	As another important ingredient to our proof, any holomorphic curve in the symplectization $\RP^{2n-1}\times \R_+$ has a well-defined index depending only on its asymptotics.  This is because $c_1(\xi_{std})$ is torsion (which is called numerically $\Q$-Gorenstein in the context of singularity theory \cite{mclean2016reeb}). When we use the  trivialization induced by the obvious disk bounded by $\gamma_i$ and $\gamma_i^2$ in $\cO(-2)$, then the SFT grading is given by\footnote{The index is computed in a similar way to \cite[Theorem 6.3]{zhou2019symplectic}. One way to explain is writing $\mu_{CZ}(\gamma^j_i)=(2i-n+1)+2j$, then $2i-n+1=\ind q_i -\frac{1}{2}\dim \CP^{n-1}$ is the Conley-Zehnder index comes from the Hamiltonian of $\frac{1}{1+\pi^*f}$ and $2j$ is the Conley-Zehnder index comes from wrapping around the disk $j$ times.}
	\begin{equation}\label{eqn:CZ}
	\mu_{CZ}(\gamma_i)+n-3=2i,\quad \mu_{CZ}(\gamma^2_i)+n-3=2i+2.
	\end{equation}
	As an example, we compare indexes from different trivializations as follows, the disk in $\cO(-2)$ bounded by $\gamma^2_0$ differs from the contraction of $\gamma^2_0$ in $\RP^{2n-1}$ by a generator $A$ of $H_2(\cO(-2))$. Therefore if we use such trivialization induced by the contraction, we have $\mu_{CZ}(\gamma^2_0)+n-3=2c_1(A)+2=2n-2$, i.e.\ $\mu_{CZ}(\gamma^2_0)=n+1$ which is same as the Conley-Zehnder index of the shortest Reeb orbits on an ellipsoid.
\end{remark}
\begin{remark}\label{rmk:CZ2}
	To elaborate the compatibility of trivializations, if we consider curve $u$ in the symplectization $\RP^{2n-1}\times \R_+$ with one positive puncture asymptotic to $\gamma_i^2$ and two negative punctures asymptotic to $\gamma_j,\gamma_k$, then the trivialization from the bounding disks in $\cO(-2)$ are compatible. This is because: if we glue the bounding disks of $\gamma_k,\gamma_i$ to $u$, we get a disk which relatively homotopic to the bounding disk of $\gamma_i^2$ in $\cO(-2)$. One can see this by checking the intersection number with $\CP^{n-1}$. On the other hand, if we change the positive asymptotics to $\gamma_i^4$, then the natural bounding disks from $\cO(-2)$ are no longer compatible with a difference from the generator of $H_2(\CP^{n-1})$.       
\end{remark}

Moreover, inside $\cH_a$, we have a partial order given by increasing homotopies. Every increasing homotopy induces a continuation map, which also preserves the splitting into $C_0$ and $C_+$. Therefore we can define the filtered symplectic cohomology as follows,
$$SH^{*,\le a}(W;\R) = \varinjlim_{H\in \cH_a} H^*(C(H)), \qquad SH^{*,\le a}_+(W;\R) = \varinjlim_{H\in \cH_a} H^*(C_+(H)).$$
And we have a tautological long exact sequence (or circle, since they are only $\Z_2$ graded in general),
\begin{equation}\label{eqn:exact}
\ldots \to H^*(W;\R)\to SH^{*,\le a}(W;\R) \to SH^{*,\le a}_+(W;\R) \to H^{*+1}(W;\R)\to \ldots.
\end{equation}
The continuation maps also gives $\iota_{a,b}:SH^{\le a}(W;\R) \to SH^{\le b}(W;\R)$ for $a\le b$, similarly for the positive symplectic cohomology. They are compatible with tautological long exact sequence. To avoid using direct limit, we will use the following.

\begin{proposition}\label{prop:iso}
	For $H\in \cH_a$, the natural morphisms $H^*(C(H))\to SH^{*,\le a}(W;\R)$ and $H^*(C_+(H))\to SH^{*,\le a}_+(W;\R)$ are both isomorphism.
\end{proposition}
\begin{proof}
We will prove that if $H\le H'\in \cH_a$, then the continuation map induces isomorphism $H^*(C(H))\simeq H^*(C(H'))$ and $H^*(C_+(H))\simeq H^*(C_+(H'))$. Note that there exists a positive number $c$ such that $H'\le H+c$. Although $H+c$ is no longer admissible in our sense, the Hamiltonian-Floer cohomology can still be defined and there is a continuation map $C(H')\to C(H+c)$. Moreover the composition $C(H)\to C(H')\to C(H+c)$ is homotopic to the continuation map from $H$ to $H+c$, which is identity. In particular, $H^*(C(H))\to H^*(C(H'))$ is injective and $H^*(C(H'))\to H^*(C(H+c))$ is surjective. We can apply the same argument to the composition $C(H')\to C(H+c) \to C(H'+c)$ to conclude that $H^*(C(H'))\to H^*(C(H+c))$ is injective. Therefore both $H^*(C(H'))\to H^*(C(H+c)$ and $H^*(C(H))\to H^*(C(H'))$ are isomorphism. Since $H^*(C_0(H))\to H^*(C_0(H'))$ is an isomorphism if we use the same Morse function $g$ on $W$, the five lemma implies that $H^*(C_+(H)) \to H^*(C_+(H'))$ is also an isomorphism.
\end{proof}

Due to the fact that $\alpha_f$ is a small perturbation of the Morse-Bott contact form $\alpha_{std}$,  we can use the Morse-Bott spectral sequence \cite[(3.2)]{seidel2006biased} to estimate the filtered positive symplectic cohomology\footnote{The statement of Morse-Bott spectral sequence in \cite{seidel2006biased} requires the vanishing of the first Chern class, the absence of this condition will not effect the validity of the spectral sequence but will cost us the grading in \cite{seidel2006biased}. But we will not use the grading in this paper.}. In fact, by a compactness and gluing argument similar to \cite{bourgeois2009symplectic}, one can show that for $\epsilon$ sufficiently small, we have $d_+ \hat{\gamma}_{i}=2\hat{\gamma}_{i-1}$, which is from the Gysin sequence of the degree $2$ circle bundle $\RP^{2n-1}\to \CP^{n-1}$. To avoid the gluing analysis overhead, in the following, we give a weaker result that is sufficient for our purpose, which only uses compactness argument in the spirit of \cite[Lemma 2.1]{cieliebak1996applications} and the Viterbo transfer map.
\begin{proposition}\label{prop:MB}
	Assume there is an exact filling $W$ of $(\RP^{2n-1},\alpha_f)$. For $\epsilon$ sufficiently small, we have the following.
	\begin{enumerate}
	\item \label{MB0} $d_+\hat{\gamma}_i=a_i\check\gamma_{i-1}$ for $a_i\ne 0$.
	\item \label{MB1} For $2<a<2+2\epsilon$, we have $\Ima(SH_+^{*,\le a}(W;\R)\to H^{*+1}(W;\R))$ is at most rank $2$ and is supported in even degrees.
\end{enumerate}
\end{proposition}
\begin{proof}
	For $\epsilon$ sufficiently small, we can find a small $\delta$ such that $\alpha_f<(1+\delta)\alpha_{std}$, i.e.\ there exists a function $h>1$ on $\RP^{2n-1}$, such that $(1+\delta)\alpha_{std}=h\alpha_f$. Then we have two exact (trivial) cobordisms $X_1$ from $(\RP^{2n-1},(1-\delta)\alpha_{std})$ to $(\RP^{2n-1},\alpha_f)$ and $X_2$ from $(\RP^{2n-1},\alpha_f)$ to $(\RP^{2n-1},(1+\delta)\alpha_{std})$. Then we have two Viterbo transfer maps from $SH^{*,\le a}_+(W;\R) \to SH^{*,\le a}_+(W^{std}_{1-\delta};\R)$ and $SH^{*,\le a}_+(W^{std}_{1+\delta};\R) \to SH^{*,\le a}_+(W;\R)$, c.f. \cite[Definition 5.2]{cieliebak2018symplectic}\footnote{Note that we need to adapt a cascades construction for the Viterbo transfer in the sense of \eqref{4} of Remark \ref{rmk:SH}, since $\alpha_{std}$ is Morse-Bott. This is just a cascades continuation map for some special Hamiltonians.}, where $W^{std}_{1\pm \delta}$ is the exact filling of $(\RP^{2n-1},(1\pm \delta)\alpha_{std})$ modified from $W$. The composition is the Viterbo transfer maps $SH^{*,\le a}_+(W^{std}_{1+\delta})\to SH^{*,\le a}_+(W^{std}_{1-\delta})$ by the functorial property of Viterbo transfer maps \cite[Proposition 5.4]{cieliebak2018symplectic}. Then for $(1+\delta)<a<2-2\delta$, we have the  $SH^{*,\le a}_+(W^{std}_{1+\delta})\to SH^{*,\le a}_+(W^{std}_{1-\delta})$ is an isomorphism by Proposition \ref{prop:Viterbo'} in the appendix. If we use $W_{\frac{1+\delta}{1-\delta}}$ to denote $W\cup \partial W\times (1,\frac{1+\delta}{1-\delta}]$, then for $\frac{1+\delta}{1-\delta}(1+(n-1)\epsilon)<a<2$, we have the Viterbo transfer $SH^{*,\le a}_+(W_{\frac{1+\delta}{1-\delta}})\to SH^{*,\le a}_+(W)$ is an isomorphism by Proposition \ref{prop:Viterbo'}. That is the two compositions in the following are both isomorphism,
	$$SH^{*,\le a}_+(W_{\frac{1+\delta}{1-\delta}})\to SH^{*,\le a}_+(W^{std}_{1+\delta})\to SH^{*,\le a}_+(W)\to SH^{*,\le a}_+(W^{std}_{1-\delta}).$$
	As a consequence, we have $SH^{*,\le a}_+(W)\to SH^{*,\le a}_+(W^{std}_{1-\delta})$ is an isomorphism. 
	
	By the same argument of \cite[Lemma 2.1]{cieliebak1996applications}, we have that all Floer trajectories as well as the continuation trajectories in the Viterbo transfer maps are contained in a tubular neighborhood of the boundary $\partial W$ (containing $\partial W^{std}_{1-\delta},\partial W^{std}_{1+\delta}$) for $\epsilon,\delta$ small enough\footnote{Since for $\epsilon=\delta=0$, all Floer trajectories/continuation trajectories become trivial cylinder on $(\RP^{2n-1},\alpha_{std})$. Note that this only uses compactness of trajectories.}. Since all curves are contained in this neighborhood, we can assign a ``local $\Z$ grading" as $c_1(\xi_{std})$ is torsion, although $c_1(W)$ may not be torsion.  Note that all of generators are in the same homotopy class, after fixing a trivialization of $\gamma_i^*\det_{\C}\xi_{std}$ and using that $c_1(\xi_{std})$ is torsion, we can assign a $\Z$-grading $|\gamma|=n-\mu_{CZ}(\gamma)$ to every generator. Moreover, we have $|\check{\gamma}_0|-|\check{\gamma}_i|=2i$ and $|\check{\gamma}_0|-|\hat{\gamma}_i|=2i+1$, e.g.\ we can use the one from Remark \ref{rmk:CZ}. The Viterbo transfer preserves this local grading as the trajectories in the Viterbo transfer are contained in this tubular neighborhood, and $SH_+^{*,\le a}(W^{std}_{1-\delta};\R)$ is the real cohomology of the critical submanifold $\RP^{2n-1}$ of parameterized simple Reeb orbits, which is supported in grading $|\check{\gamma}_0|$ and $|\hat{\gamma}_{n-1}|$. In view of the isomorphism $SH_+^{*,\le a}(W;\R)\to SH_+^{*,\le a}(W^{std}_{1-\delta};\R)$ and the grading, we must have  $d_+\hat{\gamma}_i=a_i\check\gamma_{i-1}$ for $a_i\ne 0$. If we use $\Z$-coefficient, we can conclude that $a_i=\pm 2$.
	
	In the following, we use $SH^{*,(a,b]}_+(W;\R)$ to denote the cohomology of the quotient complex generated by generators corresponding to Reeb orbits of period in $(a,b]$. Then for $2 < a < 2+2\epsilon$,  $\frac{1+\delta}{1-\delta}(2+(2n-2)\epsilon) <b<3-3\delta$ and $\frac{1+\delta}{1-\delta}(1+(n-1)\epsilon)<c<2-2\delta$, we have the following commutative diagram of long exact sequences 
	$$
	\xymatrix{
		\ldots \ar[r] & SH_+^{*,\le c}(W;\R) \ar[r]\ar[d]^{=} & SH_+^{*,\le a}(W;\R) \ar[r] \ar[d] & SH_+^{*,(c,a]}(W;\R)\ar[r]\ar[d] & \ldots \\
		\ldots \ar[r] & SH_+^{*,\le c}(W;\R) \ar[r] & SH_+^{*,\le b}(W;\R) \ar[r] & SH_+^{*,(c,b]}(W;\R)\ar[r] & \ldots }
	$$
	It is clear that $SH_+^{*,(c,a]}(W;\R)$ is the local Floer cohomology of $\overline{\gamma}_0^2$, which is generated by $\check{\gamma}_0^2,\hat{\gamma}_0^2$. By the same argument as before we have $SH_+^{*,(c,b]}(W;\R)$ is isomorphic to $H^*(\RP^{2n-1};\R)$ and generated by $\check{\gamma}_0^2, \hat{\gamma}_{n-1}^{2}$, here $\RP^{2n-1}$ is the critical submanifold in the free loop space parameterizing the space of parameterized Reeb orbits of multiplicity two of $\alpha_{std}$. In particular, the image of  $SH_+^{*,(c,a]}(W;\R)\to SH_+^{*,(c,b]}(W;\R) $ is rank one and generated by $\check{\gamma}_0^2$. As a consequence of the long exact sequence above, we have that the image of $SH_+^{*,\le a}(W;\R)\to SH_+^{*,\le b}(W;\R)$ is at most rank $3$. Since $SH^{*,\le a}_+(W;\R)\to H^{*+1}(W;\R)$ factors through $SH^{*,\le b}_+(W;\R)$, the image has rank at most $3$. On the other hand, the class represented by $\hat{\gamma}_{n-1}$ can not map to any nontrivial class in $H^*(W;\R)$, since we have $S^1$ equivariant transversality for this simple orbit, i.e.\ the Floer cylinder with positive end asymptotic to $\overline{\gamma}_{i}$ has a $S^1$-symmetry. As a result, the configuration from a hat generator $\hat{\gamma}_{n-1}$ to a critical point generator is never rigid, since we can always rotate a little to get another solution. Therefore the image of  $SH^{*,\le a}_+(W;\R)\to H^{*+1}(W;\R)$ is at most rank $2$. The last claim follows from that $\check{\gamma}_0,\check{\gamma}^2_0$ have odd degree in $C_+(H)$. 
\end{proof}

The last piece of structures we need is the pair of pants product, e.g.\ see \cite{irie2014hofer}, which is a map on the filtered symplectic cohomology,
$$\cup: SH^{\le a}(W;\R) \otimes SH^{\le b}(W;\R) \to SH^{\le a+b}(W;\R).$$
And we have the following.
\begin{proposition}\label{prop:kill}
	Let $W$ be an exact domain and $A \in \oplus_{i > 0} H^{2i}(W;\R)$. If $1+A$ is mapped to zero in $\iota_{0,a}:H^*(W;\R) \to SH^{*,\le a}(W;\R)$, then $H^*(W;\R) \to SH^{*,\le a}(W;\R)$ is zero and connecting map $SH^{*,\le a}_{+}(W;\R) \to H^{*+1}(W;\R)$ in the tautological long exact sequence is surjective\footnote{The proposition holds for $A$ with positive degree, the emphasis on even degree only makes $1+A$ degree $0$ in the $\Z_2$ grading, which we have on $SH^*(W;\R)$.}.
\end{proposition}
\begin{proof}
	We have the following commutative diagram, e.g.\ see \cite[Lemma 2.8]{irie2014hofer},
	$$
	\xymatrix{
	H^{*}(W;\R)\otimes H^{*}(W;\R) \ar[r]^{\Id \otimes \iota_{0,a}}\ar[d] & H^{*}(W;\R) \otimes SH^{\le a}(W;\R) \ar[r]^{\qquad \quad \cup} & SH^{\le a}(W;\R)\ar[d] \\
	H^{*}(W;\R)\otimes H^{*}(W;\R) \ar[r]^{\qquad \cup} & H^{*}(W;\R) \ar[r]^{\iota_{0,a}} & SH^{\le a}(W;\R)}
	$$
	Since $A$ is nilpotent in $H^*(W;\R)$, we have that $1+A$ is a unit in $H^*(W;\R)$. Then for any $x\in H^*(W;\R)$, we have $0=(x\cup (1+A)^{-1})\cup \iota_{0,a}(1+A)=\iota_{0,a}(x)$, i.e.\ $H^{*}(W;\R)\to SH^{*,\le a}(W;\R)$ is zero. Then by the tautological long exact sequence, we have $SH^{*,\le a}_+(W;\R) \to H^{*+1}(W;\R)$ is surjective.
\end{proof}

\subsection{Vanishing of symplectic cohomology}
The key ingredient in our proof is that $\check{\gamma}^2_0$ up to certain error kills the symplectic cohomology as it does for the double cover whenever $n\ge 3$. If this was proven, Proposition \ref{prop:kill} can be used to estimate the rank of cohomology of the filling by filtered positive symplectic cohomology. To show that  $\check{\gamma}^2_0$ essentially kills the unit, we will study the map $SH^{*,\le 2+\epsilon}_+(W;\R)\stackrel{d_{+,0}}{\longrightarrow} H^{*+1}(W;\R) \stackrel{\text{projection}}{\longrightarrow} H^{0}(W;\R)$, in particular, we are interested to understand if $1$ is in the image, i.e.\ we are interested in the contribution  $d_{+,0}(\check{\gamma}^2_0)$ to the minimum of $g$. The choice of action threshold is to include $\check{\gamma}^2_0$ in the cochain complexes but nothing else with larger periods. 

By our setup of symplectic cohomology, one part of this contribution is the count of the following moduli space (i.e.\ $1$ level cascades).
\begin{equation}\label{eqn:disk}
\cM(\check{\gamma}^2_0, q):=\{u:\C \to \widehat{W}|\partial_s u+J_t(\partial_t u-X_H)=0, \lim_{s\to \infty} u(s,0)=\check{\gamma}^2_0, u(0)=q\}/\R,\end{equation}
where $q$ is a fixed point inside $W$, which is the unique minimum of the Morse function $g$ and $(s,t)$ is the polar coordinate on $\C^*$ by $(s,t)\mapsto e^{2\pi(s+it)}$. Since $q$ is the minimum and $\check{\gamma}^2_0$ is a check generator, both the Morse flow lines degenerate to point constraints. Since $\check{\gamma}^2_0$ is on $\Ima \overline{\gamma}^2_0$, we will call the constraint from $\check{\gamma}^2_0$ an orbit point constraint to differentiate it from the point constraint from $q$. We can choose $q$ to be arbitrarily close to $\partial W$. We will perform a neck stretching along $Y_1\subset W$, which is a slight push-in along the $-r$ direction and strictly contactomorphic to $(\RP^{2n-1}, (1-\delta)\alpha_f)$ for $\delta$ small. We use $X$ to denote the cobordism from $(\RP^{2n-1}, (1-\delta)\alpha_f)$ to $(\RP^{2n-1}, \alpha_f)$. We may assume $q$ is an interior point of $X$. We use $\partial_+ X$ to denote the positive boundary and $\partial_- X$ to denote the negative boundary.

We first recall the setup of neck-stretching for general case following \cite[\S 3.2]{zhou2019symplectic}.  Let $(W,\lambda)$ be a exact domain and $(Y_1,\alpha:=\lambda|_{Y_1})$ be a contact type hypersurface inside $W$. The hyperplane divides $W$ into a cobordism $X$ union with a domain $W'$. Then we can find a small slice $(Y_1\times [1-\eta,1+\eta]_r,\rd(r\alpha))$ symplectomorphic to a neighborhood of $Y_1$ in $W$. Assume $J|_{Y_1\times [1-\eta,1+\eta]_{r}}=J_0$, where $J_0$ is independent of $S^1$ and $r$ and $J_0(r\partial_{r})=R_\alpha,J_0\xi=\xi$ for $\xi:=\ker \alpha$. Then we pick a family of diffeomorphism $\phi_R:[(1-\eta)e^{1-\frac{1}{R}}, (1+\eta)e^{\frac{1}{R}-1}]\to [1-\eta,1+\eta]$ for $R\in (0,1]$ such that $\phi_1=\Id$ and $\phi_R$ near the boundary is linear with slope $1$. Then the stretched almost complex structure $NS_{R}(J)$ is defined to be $J$ outside $Y_1\times [1-\eta,1+\eta]$ and is $(\phi_R\times \Id)_*J_0$ on $Y_1\times [1-\eta,1+\eta]$. Then $NS_{1}(J)=J$ and $NS_{0}(J)$ gives almost complex structures on the completions $\widehat{X}$, $\widehat{W'}$ and $Y_1\times \R_+$, which we will refer as the fully stretched almost complex structure.

We will consider the degeneration of curves solving the Floer equation with one positive cylindrical end asymptotic to a non-constant Hamiltonian orbit of $X_H$. Since either the orbit is simple or $J$ depends on the $S^1$ coordinate near non-simple orbits, the topmost curve in the SFT building, i.e.\ the curve in $\widehat{X}$, has the somewhere injectivity property. In particular, we can find regular $J_t$ on $\widehat{X}$ such that all relevant moduli spaces, i.e.\ those with point constraint from $q$ (which is in $\widehat{X}$), or with negative cylindrical ends asymptotic to non-constant Hamiltonian orbits of $X_H$, possibly with negative punctures asymptotic to Reeb orbits of $Y_1$ and multiple cascades levels, are cut out transversely. We say a almost complex structure on $W$ is generic iff the fully stretched almost complex structure $NS_0(J)$ is regular on $\widehat{X}$. The set of generic almost complex structures form an open dense subset \footnote{This is because there are only finitely many moduli spaces that can have positive energy.} in the set of compatible almost complex structures  that are cylindrical convex and $S^1$, $r$ independent on $Y_1\times [1-\eta,1+\eta]_r$. 

For the compactification of curves in the topmost SFT level, in addition to the usual SFT building in the symplectization $\partial_+X\times \R_+=Y_1\times \R_+$ stacked from below \cite{bourgeois2003compactness}, we also need to include Hamiltonian-Floer breakings near the cylindrical ends. In our context, since we use autonomous Hamiltonians and cascades,  we need to include curves with multiple cascades levels and their degeneration, e.g.\ $l_i=0,\infty$ in the cascades for some horizontal level $i$.  A generic configuration is described in the top-right of the figure below, but we could also have more cascades levels with the connecting Morse trajectories degenerate to $0$ length or broken Morse trajectories.
\begin{figure}[H]
	\begin{center}
		\begin{tikzpicture}[scale=0.5]
		\path [fill=blue!15] (0,0) to [out=20, in=160]  (6,0) to [out=270,in=90] (6,-6) to [out=170,in=10] (0,-6) to [out=90, in=270] (0,0);
		\path [fill=red!15] (0,-6) to [out=10, in=170]  (6,-6) to [out=270,in=0] (3,-10) to [out=180,in=270] (0,-6); 
		\draw (0,0) to [out=20, in=160]  (6,0) to [out=270,in=90] (6,-6) to [out=170,in=10] (0,-6) to [out=90, in=270] (0,0);
		\draw (6,-6) to [out=270,in=0] (3,-10) to [out=180,in=270] (0,-6); 
		\draw[dashed] (0, -5.5) to [out=10, in=170] (6,-5.5);
		\draw[dashed] (0.02,-6.5) to [out=10, in=170] (5.98,-6.5);
		\draw[->] (1,-1) to (1.5,-1);
		\draw (1.5,-1) to (2,-1);
		\draw (2,-1) to [out=90, in=180] (2.5, -0.75) to [out=0, in = 90] (3,-1) to [out=270,in=0] (2.5,-1.25) to [out=180,in=270] (2,-1);
		\draw (2,-1) to [out=270,in=90] (1,-7) to [out=270,in=180] (2,-8) to [out=0,in=180](2.5,-3) to [out=0, in=90](3,-4);
		\draw (3,-1) to [out=270, in=90] (4,-4) to [out=270, in=0] (3.5,-4.25) to [out=180,in=270] (3,-4);
		\draw[dotted] (3,-4) to [out=90, in=180] (3.5,-3.75) to [out=0, in=90] (4,-4);
		\draw[->] (4,-4) to (4.25,-4);
		\draw (4.25,-4) to (4.5,-4);
		\draw (4.5,-4) to [out=90, in=180] (5,-3.75)  to [out=0, in=90] (5.5,-4) to [out=270, in=0] (5,-7) to [out=180, in=270] (4.5,-4) to [out=270,in=180] (5,-4.25) to [out=0,in=270] (5.5,-4);
		\fill (5,-5) circle[radius=1pt];
		\node at (5.3,-5) {$q$};
		\node at (2.5,-2) {$u_1$};
		\node at (5,-4.5) {$u_2$};
		\end{tikzpicture}
		\hspace{1cm}
		\begin{tikzpicture}[xscale=0.5,yscale=0.7]
		\path [fill=blue!15] (0,-1) to [out=20, in=160]  (6,-1) to [out=270,in=90] (6,-6) to [out=170,in=10] (0,-6) to [out=90, in=270] (0,-1);
		\path [fill=red!15] (0,-6) to [out=10, in=170]  (6,-6) to [out=270,in=0] (3,-10) to [out=180,in=270] (0,-6); 
		\draw (0,-1) to [out=20, in=160]  (6,-1) to [out=270,in=90] (6,-6) to [out=170,in=10] (0,-6) to [out=90, in=270] (0,-1);
		\draw (6,-6) to [out=270,in=0] (3,-10) to [out=180,in=270] (0,-6);
		\draw[dashed] (0, -5) to [out=10, in=170] (6,-5);
		\draw[dashed] (0.1, -7) to [out=10, in=170] (5.9,-7);
		\draw[->] (1,-2) to (1.5,-2);
		\draw (1.5,-2) to (2,-2);
		\draw (2,-2) to [out=90, in=180] (2.5, -1.8) to [out=0, in = 90] (3,-2) to [out=270,in=0] (2.5,-2.2) to [out=180,in=270] (2,-2);
		\draw (2,-2) to [out=270,in=90] (1,-7) to [out=270,in=180] (2,-8) to [out=0,in=180](2.5,-3.5) to [out=0, in=90](3,-4);
		\draw (3,-2) to [out=270, in=90] (4,-4) to [out=270, in=0] (3.5,-4.2) to [out=180,in=270] (3,-4);
		\draw[dotted] (3,-4) to [out=90, in=180] (3.5,-3.8) to [out=0, in=90] (4,-4);
		\draw[->] (4,-4) to (4.25,-4);
		\draw (4.25,-4) to (4.5,-4);
		\draw (4.5,-4) to [out=90, in=180] (5,-3.8)  to [out=0, in=90] (5.5,-4) to [out=270, in=0] (5,-7) to [out=180, in=270] (4.5,-4) to [out=270,in=180] (5,-4.2) to [out=0,in=270] (5.5,-4);
		\fill (5,-4.7) circle[radius=1pt];
		\node at (5.3,-4.8) {$q$};
		\node at (2.5,-3) {$u_1$};
		\node at (5,-4.5) {$u_2$};
		\end{tikzpicture}
		\hspace{1cm}
		\begin{tikzpicture}[scale=0.5]
		\path [fill=blue!15] (0.5,-6) to [out=90, in=270]  (0,0) to [out=20,in=160] (6,0) to [out=270,in=90] (5.5,-6) to [out=160, in=20] (0.5,-6);
		\path [fill=purple!15] (0.5,-6.2) to [out=20,in=160] (5.5,-6.2) to [out=270,in=90] (5.5,-12) to [out=160, in=20] (0.5,-12) to [out=90, in=270] (0.5,-6.2);
		\path [fill=red!15] (0.5, -12.2) to [out=20,in=160] (5.5,-12.2) to [out=270,in=0] (3,-16) to [out=180,in=270] (0.5,-12.2);
		\draw (0.5,-6) to [out=90, in=270]  (0,0) to [out=20,in=160] (6,0) to [out=270,in=90] (5.5,-6);
		\draw [dashed] (5.5,-6) to [out=160, in=20] (0.5,-6);
		\draw[dashed] (0.5,-6.2) to [out=20,in=160] (5.5,-6.2);
		\draw (5.5,-6.2) to [out=270,in=90] (5.5,-12);
		\draw[dashed] (5.5,-12) to [out=160, in=20] (0.5,-12);
		\draw (0.5,-12) to [out=90, in=270] (0.5,-6.2);
		\draw [dashed](0.5, -12.2) to [out=20,in=160] (5.5,-12.2); 
		\draw (5.5,-12.2)to [out=270,in=0] (3,-16) to [out=180,in=270] (0.5,-12.2);
		\draw[->] (1,-1) to (1.5,-1);
		\draw (1.5,-1) to (2,-1);
		\draw (2,-1) to [out=90, in=180] (2.5, -0.75) to [out=0, in = 90] (3,-1) to [out=270,in=0] (2.5,-1.25) to [out=180,in=270] (2,-1);
		\draw (2,-1) to [out=270,in=90] (1,-4) to [out=270, in=180]  (1.5,-5.8) to [out=0,in=270](2,-5) to [out=90,in=180] (2.5,-3) to [out=0, in=90](3,-4);
		\draw (3,-1) to [out=270, in=90] (4,-4) to [out=270, in=0] (3.5,-4.25) to [out=180,in=270] (3,-4);
		\draw[dotted] (3,-4) to [out=90, in=180] (3.5,-3.75) to [out=0, in=90] (4,-4);
		\draw[->] (4,-4) to (4.25,-4);
		\draw (4.25,-4) to (4.5,-4);
		\draw (4.5,-4) to [out=90, in=180] (4.9,-3.75)  to [out=0, in=90] (5.3,-4) to [out=270, in=0] (4.9,-5.9) to [out=180, in=270] (4.5,-4) to [out=270,in=180] (4.9,-4.25) to [out=0,in=270] (5.3,-4);
		\fill (5,-5.2) circle[radius=1pt];
		\node at (5.1,-5.5) {$q$};
		\draw (1.5,-5.9) to [out=180,in=90] (0.8,-9) to [out=270, in=180] (1.5,-11.8) to [out=0,in=270] (2.2,-9) to [out=90, in=0] (1.5,-5.9);
		\draw (1.5,-11.9) to [out=180,in=90] (0.8,-13) to [out=270, in=180] (1.5,-14) to [out=0,in=270] (2.2,-13) to [out=90, in=0] (1.5,-11.9);
		\draw (5,-6) to [out=180,in=90] (4.6,-8) to [out=270, in=180] (5,-9) to [out=0,in=270] (5.3,-8) to [out=90, in=0] (5,-6);
		\node at (1.5,-5.8) [circle, fill=white, draw, outer sep=0pt, inner sep=3 pt] {};
		\node at (4.9,-5.9) [circle, fill=white, draw, outer sep=0pt, inner sep=3 pt] {};
		\node at (1.5,-11.8) [circle, fill=white, draw, outer sep=0pt, inner sep=3 pt] {};
		\node at (2.5,-2) {$u^{\infty}_1$};
		\node at (5,-4.7) {$u^{\infty}_2$};
		\node at (4.5,-2) {$\widehat{X}$};
		\node at (4,-10) {$Y\times \R_+$};
		\node at (4,-14) {$\widehat{W'}$};
		\end{tikzpicture}
	\end{center}
    \caption{Neck-stretching}
    \label{fig}
\end{figure}
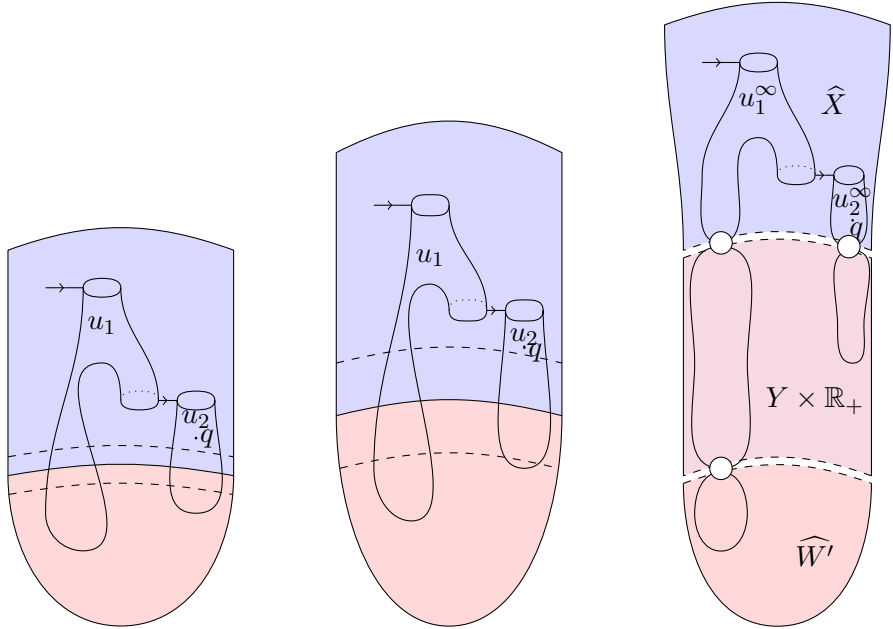
In the figure, we use $\bigcirc$ to indicate the puncture that is asymptotic to a Reeb orbit. We call $u_1$ is the topmost level of the cascades and all curves in $\widehat{X}$ in the fully stretched picture curves in the topmost (SFT building) level. We call $u_1^{\infty}$ the topmost cascades level in the topmost level. $u_2$ is the bottom cascades level and $u_2^\infty$ is the bottom cascades level in the topmost SFT level. 

The benefit of neck-stretching is two-fold. (1) After neck-stretching along $Y_1$ inside the hypothetical exact filling $W$ of $(\RP^{2n-1},\alpha_f)$, the virtual dimension of the topmost level can be computed using only their asymptotic orbits. This is because $c_1(\xi_{std})$ is torsion. While we can not do this in the filling, since $c_1(W)$ is not necessarily zero in $H^2(W;\Q)$, in particular, we need to keep track of the relative homology class of the curve. (2) In the topmost level, $\gamma^2_0,\gamma_0$ are in different homology classes, which may not be the case in the filling $W$. The price we pay is that we need to analyze more configurations.

In the fully stretched case, one particularly important moduli space is the following $1$-level cascades
\begin{equation}\label{eqn:SFT}
\cM^{\infty}(\check{\gamma}^2_0,q):=\{u:\C \to \widehat{X}|\partial_s u+J_t(\partial_t u-X_H)=0, \lim_{s\to \infty} u(s,0)=\check{\gamma}^2_0, u(0)=q\}/\R,
\end{equation}
which is closely related to \eqref{eqn:disk}, as we shall see in Proposition \ref{prop:unit}. In the following, we derive an action constraint. Let $\tilde{\omega}$ denote the smooth $2$-form on $\widehat{X}$, such that it is the symplectic form on $X$ and $\partial_+X\times (1,\infty)$ and $\rd(g\alpha_f)$ on $\partial_-X\times (0,1)$, where $g$ is a strictly increasing positive function on $(0,1)$ such that $\lim_{x\to 0}g(x)=(1-\delta)\eta$ and $\lim_{x\to 1} g=(1-\delta)$ for $\eta<1$. It is clear that we can find such $g$, so that $\tilde{\omega}$ is smooth and exact on $\widehat{X}$ with a smooth primitive $\tilde{\alpha}$ which is $r \alpha_f$ on $X\cup \partial_+X\times (1,\infty)$ and $g\alpha_f$ on $\partial X_-\times (0,1)$. Since $J_t$ is cylindrical convex on $\partial_-X \times (0,1)$, $J_t$ is also compatible with $\rd(g \alpha_f)$ on $\partial_-X\times (0,1)$. Then for any $u$ solving the Floer equation in \eqref{eqn:SFT} but possibly with negative punctures asymptotic to Reeb orbits on $\partial X_-$, we have $\tilde{\omega}(\partial_su,\partial_tu-X_H)\ge 0$. We use $\widetilde{X_H}$ to denote the Hamiltonian vector field of $H$ using $\widetilde{\omega}$. Since $H=0$ below  $\partial_-X$, we have $X_H=\widetilde{X_H}$. Then by Stokes' theorem, the integration of $\tilde{\omega}(\partial_su,\partial_tu-X_H)=\tilde{\omega}(\partial_su,\partial_tu-\widetilde{X_H})$ implies 	\begin{eqnarray*}
	0 & \le & \int_{\C\backslash\{z_1,\ldots, z_{|\Gamma|\}}} \tilde{\omega}(\partial_su,\partial_tu-\widetilde{X_H})\rd s \wedge \rd t \\
	& = & \int_{\C\backslash\{z_1,\ldots, z_{|\Gamma|\}}}\rd (u^*\tilde{\alpha})-\int_{\C}\partial_s u^*H  \rd s\wedge \rd t\\
	& = & 2r-\sum_{\gamma \in \Gamma }(1-\delta)\eta \int_{\gamma} \alpha_f-h(r)
\end{eqnarray*}
where $r$ is the unique value such that $h'(r)=2$ and $\Gamma$ is set of negative asymptotic orbits of $u$ viewed as the Reeb orbits of $\alpha_f$.\footnote{A priori, the negative puncture of $u$ is asymptotic to Reeb orbits of $(\RP^{2n-1},(1-\delta)\alpha_f)=Y_1=\partial_-X$. That is why we have $(1-\delta)$ in the expression when we view the Reeb orbits as in $(\RP^{2n-1},\alpha_f)$.} The first equality follows from that $u^*H$ is zero near punctures $z_i$. Let $\eta \to 1$, we have
\begin{equation}
\underbrace{2r-h(r)}_{\substack{\text{negative symplectic} \\ \text{action of } \overline{\gamma}^2_0}}-\underbrace{\sum_{\gamma \in \Gamma }(1-\delta)\int_{\gamma} \alpha_f}_{\text{contact action of }\Gamma}\ge 0,
\end{equation}
When the width of $H$ converges to zero, the unique value $r$ such that $h'(r)=2$ converges to $1$. In the meantime, $h(r)$ converges to $0$. Therefore if we choose $H$ to have arbitrarily small width and $\delta$ arbitrarily small, we have $2r-h(r)\to 2$ and $$\sum_{\gamma \in \Gamma} \int_{\gamma}\alpha_f < 2+\epsilon'=\int_{\gamma_0^2}\alpha_f+\epsilon',$$
for $\epsilon'>0$ sufficiently small. Let $\cC(\gamma)=\int_{\gamma}\alpha_f$ denote the contact action. In general, for a curve $u$ in $\widehat{X}$, possibly with multiple cascades levels, with topmost positive cylindrical end asymptotic to $\overline{\gamma}_+$ and bottom negative cylindrical end asymptotic to $\overline{\gamma}_-$ and a collection of negative punctures asymptotic to $\Gamma$, in addition to the usual symplectic action relation $\cA(\overline{\gamma}_+)<\cA(\overline{\gamma}_-)$, we also have
\begin{equation}\label{eqn:action}
\cC(\gamma_+)-\cC(\gamma_-)\ge \sum_{\gamma\in \Gamma}\cC(\gamma),
\end{equation}
for suitable choice of $H$ and $\delta$  by the same argument as above.
Of course, the choice of $H$ and $\delta$ depends on $\overline{\gamma}_+,\overline{\gamma}_-$. In the case of $H\in \cH_a$, there are finitely many families of orbits. In the following we fix $H$ and $\delta$ such that \eqref{eqn:action} holds for any $\overline{\gamma}_+,\overline{\gamma}_-$ in our setup of symplectic cohomology and neck-stretching. We refer to the $\cC(\gamma_+)-\cC(\gamma_-)- \sum_{\gamma\in \Gamma}\cC(\gamma)$ as the contact energy, which is non-negative. 

In the following, we first state a key property for the double cover $(S^{2n-1},\xi_{std})$, which will supply us with the holomorphic curve we need for $(\RP^{2n-1},\xi_{std})$ for Proposition \ref{prop:unit}. The following result follows from a tailored proof of \cite[Theorem A]{zhou2019symplectic}.
\begin{proposition}\label{prop:curve}
	Let $(S^{2n-1},\alpha)$ be the standard contact sphere with a non-degenerate ellipsoid contact form that is close to a round sphere with $n\ge 2$. Then for any small enough positive number $\delta$. Let $X'$ denote the (trivial) symplectic cobordism from $(S^{2n-1},(1-\delta)\alpha)$ to $(S^{2n-1},\alpha)$. Let $q$ be a interior point in $X'$, let $H$ be an admissible Hamiltonian on the completion $\widehat{X'}$ in the sense of last subsection, in particular, $H$ is zero below $(S^{2n-1},\alpha)$ and is linear on the positive end of $\widehat{X'}$. Let $\gamma$ be the Reeb orbit with minimal period of $\alpha$, we define
	\begin{equation}\label{eqn:sphere}
	\cM^{\infty}(\check{\gamma},q):=\{u: \C \to \widehat{X'}|\partial_s u + J_t(\partial_tu-X_H)=0,\lim_{s\to\infty}u(s,0)=\check{\gamma},u(0)=q\}/\R.
	\end{equation}
	We say $J_t$ is nice, if every curve in $\cM^{\infty}(\check{\gamma},q)$ is cut out transversely and there is no curve in form of those in $\cM^{\infty}(\check{\gamma},q)$ with one extra negative puncture asymptotic to a simple Reeb orbit. Then the set of nice $J_t$ is not empty, and for any nice $J_t$, $\cM^{\infty}(\check{\gamma},q)$ is compact and the algebraic count is $1$ after we choose an appropriate orientation of the determinant line bundle associated to $\overline{\gamma}$. 
\end{proposition}
\begin{proof}
	Assume $J_t$ is nice but the moduli space is not compact, then we have a SFT building breaking. First of all, there are no multiple cascades level in the topmost SFT level, i.e.\ no configuration in the fully stretched case of Figure \ref{fig}. This is because $\overline{\gamma}$ already has the maximal symplectic action. If there was a multi-level cascades, the negative cylindrical end must be asymptotic to a non-constant Hamiltonian orbits with larger symplectic action, which is impossible. Then we only need to rule out the case of $1$-level cascades with negative punctures for the topmost SFT level. By action reasons explained in \eqref{eqn:action}, there is at most one negative puncture asymptotic to a simple Reeb orbit, but such configuration is ruled out since $J_t$ is nice. 
	
	Next we will show the set of nice $J_t$ is not empty, in fact, a generic $J_t$ is nice. Since $(1-\delta)\alpha$ is a non-degenerate ellipsoid, the minimal Conley-Zehnder index is $n+1$. Then the virtual dimension of the topmost curve, i.e.\ a curve in $\cM^{\infty}(\check{\gamma},q)$ with possible negative punctures asymptotic to $\Gamma_-$,  is $-\sum_{\gamma_-\in \Gamma_-}(\mu_{CZ}(\gamma_-)+n-3)\le -(2n-2)$ as long as $\Gamma_-\ne \emptyset$, where $\Gamma_-$ is the set of negative asymptotic Reeb orbits of the topmost curve. Since the transversality of the topmost curve is guaranteed by the genericity of $J_t$, there is no such SFT building. 
	
	To prove the algebraic count is $1$, we consider the filling of $(S^{2n-1},(1-\delta)\alpha)$ by the standard ball, which union with the cobordism $X$ is a filling $D$ of $(S^{2n-1},\alpha)$, i.e.\ the standard ball. We can use $H$ and a Morse function $g$ with unique minimum at $q$ to define symplectic cohomology of $D$.\footnote{A priori, $H$ has a finite slope, hence only defines a filtered symplectic cohomology. However we can modify $H$ outside a large $r$ to be a small perturbation of the quadratic Hamiltonian, which would define the full symplectic cohomology. Since we are only interested in the moduli space asymptotic to an interior point and $\overline{\gamma}$. The integrated maximal principle implies that any change outside a large region does not affect our curve.} Since $SH^*(D)=0$ with $\Z$ coefficient, and there is only one generator with degree $-1$, that is exactly $\check{\gamma}$. Therefore we must have $d_{+,0}(\check{\gamma})=\pm q$, and the coefficient can be fixed to $1$ after we choose an appropriate orientation of the determined line of $\overline{\gamma}$. Note that $d_{+,0}(\check{\gamma})=q$ implies that 
	\begin{equation}\label{eqn:1}
	\# \{u:\C \to \widehat{D}|\partial_su+J_t(\partial_t u-X_H)=0, \lim_{s\to \infty} u(s,0)=\check{\gamma},u(0)=q \}/\R=1,
	\end{equation}
	for any regular admissible $J_t$. Then we can apply neck-stretching along $(S^{2n-1},(1-\delta)\alpha)$. If the fully stretched almost complex structure is nice, then we have the moduli space \eqref{eqn:1} is contained completely outside $(S^{2n-1},(1-\delta)\alpha)$ and regular for sufficiently stretched almost complex structure, since there can not be any breaking. And in the fully stretched case, it is identified with \eqref{eqn:sphere}. Hence the claim follows.
\end{proof}	

In the following, we use $\la d_{+,0}(a), b \ra$ to denote the coefficient of $b$ in $d_{+,0}(a)$. Since only bottom cascades level can have the point constraint $u(0)=q$, which makes the bottom cascades level has a relative low virtual dimension, we will focus on analyzing the bottom cascades level in the following proposition.
\begin{proposition}\label{prop:unit}
	For sufficiently stretched generic almost complex structure $J_t$, we have the algebraic count of $\cM(\check{\gamma}^2_0, q)$ is $2$ and $\la d_{+,0}(\check{\gamma}^2_0), q\ra = 2$.
\end{proposition}
\begin{proof}
	The proof is divided into three parts.
	
	\textbf{Step 1: } For sufficiently stretched generic almost complex structure $J_t$, we have $\cM(\check{\gamma}^2_0,q)=\cM_{\check{\gamma}^2_0,q}$.

	We need to rule out multiple level cascades in order to prove the equality. Suppose we have a multiple level cascade, since each curve increases the symplectic action, then the negative cylindrical ends of the topmost cascades level is asymptotic to an orbit in $\overline{\gamma}_i$ for some $i$. If we apply neck-stretching to $Y_1=\partial_-X$, since $\gamma^2_0$ and $\gamma_i$ are in different homology classes, we must have an extra negative puncture in the limit of neck-stretching. Then by the action reason \eqref{eqn:action}, we must have $i=0$ for sufficiently stretched $J_t$. Therefore the bottom level of the cascades is a map $u:\C \to \widehat{W}$ with $\partial_su+J_t(\partial_t u -X_H)=0$ and $\lim_{s\to \infty}u(s,\cdot)\in\overline{\gamma}_0$ and $u(0)=q$. Then in the full neck-stretching, by action and homology class reason, we end up a map with an extra negative puncture asymptotic to $\gamma_0$.  Since $c_1(\xi_{std})$ is torsion and the trivializations from the obvious disk in $\cO(-2)$ are compatible by Remark \ref{rmk:CZ2}, the virtual dimension of such space (the positive cylindrical end has no orbit point constraint)  is 
	$$\mu_{CZ}(\gamma_0)-n-(\mu_{CZ}(\gamma_0)+n-3)=\mu_{CZ}(\hat{\gamma}_0)-n-1-(\mu_{CZ}(\gamma_0)+n-3)=3-2n<0.$$
	Therefore for generic and sufficient stretched $J_t$, there is no such configuration hence no multi-level cascades.

	\textbf{Step 2: }For sufficiently stretched generic almost complex structure $J_t$, we have $\cM^{\infty}(\check{\gamma}^2_0,q)$ is identified with $\cM(\check{\gamma}^2_0,q)$ and both of them are compact.
	
	We first argue for generic fully stretched $J_t$, $\cM^{\infty}(\check{\gamma}^2_0,q)$ is compact. By the same argument in step 1, there is no multi-level cascades like Figure \ref{fig} in the compactification of $\cM^{\infty}(\check{\gamma}^2_0,q)$. Then we need to rule out the case of $1$-level cascades with negative punctures. The virtual dimension of the moduli space of curves solving \eqref{eqn:SFT} with negative punctures asymptotic asymptotic to Reeb orbits in $\Gamma$ is the following
	\begin{equation}\label{eqn:index}
	\underbrace{\mu_{CZ}(\check{\gamma}^2_0)-n-1}_{\text{virtual dimension of \eqref{eqn:SFT}}}-\sum_{\gamma \in \Gamma}(\mu_{CZ}(\gamma)+n-3)= (\mu_{CZ}(\check{\gamma}^2_0)+n-3)-\sum_{\gamma \in \Gamma}(\mu_{CZ}(\gamma)+n-3) +2-2n.
	\end{equation}
	We have to make sure the Conley-Zehnder index are computed using compatible trivializations. By action reason explained above and homology class of the Reeb orbits, we know the only SFT building breaking configurations for the compactification of $\cM^{\infty}(\check{\gamma}^2_0,q)$ that we can have contain either two negative punctures both asymptotic to $\gamma_0$ or one negative puncture asymptotic to $\gamma^2_0$.  The trivializations from the obvious disk in $\cO(-2)$ are compatible, hence the virtual dimensions are well-defined and they are $4-2n<0$, $2-2n<0$ respectively. That is such configuration will not appear for generic $J_t$. That is $\cM^{\infty}(\check{\gamma}^2_0,q)$ is compact for generic $J_t$. By the similar argument in Proposition \ref{prop:curve} and the dimension computation above, we know that for sufficiently stretched generic $J_t$, $\cM(\check{\gamma}^2_0,q)$ is contained outside $\partial X_-$ and is identified with $\cM^{\infty}(\check{\gamma}^2_0,q)$.

	\textbf{Step 3: }For generic almost complex structure $J_t$, we have  $\#\cM^{\infty}(\check{\gamma}^2_0,q)=2$.

	Let $\gamma$ be the lift of $\gamma^2_0$ in $(S^{2n-1},p^*\alpha_f)$. Let $\check{\gamma}^\pm$, $q^{\pm}$ denote the two lifts of $\check{\gamma}_0^2$ and $q$ in $S^{2n-1}$. Then it is clear that we have a map
	$$P:\cM:=\bigcup_{\Diamond,\heartsuit\in \{\pm \}}\cM^\infty(\check{\gamma}^{\Diamond},q^{\heartsuit})\to \cM^{\infty}(\check{\gamma}_0^2,q),$$
	induced by the projection $p:S^{2n-1}\to \RP^{2n-1}$. Since $\C$ is simply connected, every curve in $\cM^{\infty}(\check{\gamma}_0^2,q)$ has two lifts in $\cM$. Therefore $P$ is a two-to-one surjective map. We know that $p^*\alpha_f$ is a non-degenerate ellipsoid close to a round sphere, to apply Proposition \ref{prop:curve}, we need to show that $p^*J_t$ is nice. First we verify every curve in $\cM$ is cut out transversely. It is clear that a non-zero vector in the kernel of the linearized perturbed Cauchy-Riemann operator $D_2$ for $\cM$ will  be pushed down by $p_*$ (note that $p$ is a local diffeomorphism) to a non-zero vector in the kernel of the linearized perturbed Cauchy-Riemann operator $D_1$ for $\cM^{\infty}(\check{\gamma}_0^2,q)$. Therefore $\dim \ker D_2\le \dim \ker D_1$. Since $J_t$ is generic, we have $\dim \ker D_1=1$ generated by the $\R$ translation. Hence $\dim \ker D_2=1$, since $\ker D_2$ always has the vector generated by the $\R$ translation. Since both the expected dimension of $\cM$ and $\cM^{\infty}(\check{\gamma}_0^2,q)$ are zero, we have $\dim \coker D_2=\dim \coker D_1=0$, i.e.\ $\cM$ is cut out transversely. To prove $p^*J_t$ is nice, we still need to show that there is no curve with one extra negative puncture asymptotic to a simple Reeb orbit. Any such curve can be pushed to $\RP^{2n-1}$ via $p$ to a curve with a negative puncture asymptotic to $\gamma_i^2$. However, such configuration is ruled out in the previous step for generic $J_t$. 
	
	As a consequence, by Proposition \ref{prop:curve}, we know that each of the four components of $\cM$ has an algebraic count of $1$ with an appropriate orientation of the determinant line bundle over $\overline{\gamma}$. This choice of orientation is consistent for all four components of $\cM$ as $\check{\gamma}^{\pm}, q^{\pm}$ are connected to each other respectively in the space of (orbit) point constraints. We can push the orientation of the determinant bundle of  $\overline{\gamma}$ to an orientation of the determinant bundle of  $\overline{\gamma}_0^2$ because $\gamma_0^2$ is a good orbit. Using this orientation structure for $\cM^{\infty}(\check{\gamma}_0^2,q)$, we know that $P$ preserves orientations and $\#\cM=4$. Therefore we have
	$\#\cM^{\infty}(\check{\gamma}^2_0,q)=2$. This finishes the proof of the proposition, since $\la d_{+,0}(\check{\gamma}^2_0),q\ra =\# \cM_{\check{\gamma}^2_0,q}=\# \cM(\check{\gamma}^2_0,q)=\# \cM^\infty(\check{\gamma}^2_0,q)$.
\end{proof}
\begin{remark}
	Here we use $n\ge 3$ to rule out the other potential configuration from neck-stretching in step 2. However this is just a convenient argument and it is not the reason our proof breaks down when $n=2$. In fact, if we use a pure symplectic field theory setup, then the curve is necessarily a double cover of a trivial cylinder, that lives over the critical point $q_0$ of $f$. Then we choose $q$ such that $\pi(q)\ne q_0$, then there is no such configuration. 
\end{remark}

\begin{proposition}\label{prop:other}
	When $n\ge 3$, for a generic and sufficiently stretched almost complex structure, we have $\la d_{+,0}\check{\gamma}_1,q\ra=0$.
\end{proposition}
\begin{proof}
	We first argue that for generic and sufficiently stretched almost complex structure, there is no contribution from multiple level cascades. If there is a multiple level cascades, then by symplectic action reason, the topmost cascades' negative end must be asymptotic to $\overline{\gamma}_0$. As a consequence, the bottom level of the cascades must have positive cylindrical end asymptotic to $\overline{\gamma}_0$. However such configuration was ruled out in the step 1 of Proposition \ref{prop:unit}.

	Next we argue that it is impossible to have a single level cascades contributing to $\la d_{+,0}\check{\gamma}_1,q\ra$. Since $\gamma_1$ is not contractible in $\RP^{2n-1}$, we know that in the fully stretched configuration, we must have breaking into holomorphic buildings. Since $q$ is in $X$ and by contact action reasons,  the topmost curve in the SFT building must be a curve $u:\C\backslash \{ z_0\} \to \widehat{X}$ such that 
	$$\partial_s u+J_t(\partial_t u-X_H)=0, \lim_{s\to \infty} u(s,0)=\check{\gamma}_1, u(0)=q,$$
	with $z_0$ is a negative puncture, where $u$ is asymptotic to $\gamma_0$ or $\gamma_1$. Since the trivializations from the disks in $\cO(-2)$ are compatible for our moduli space,  we have the moduli space of the above curve has a well-defined virtual dimension $4-2n$ for $\gamma_0$ puncture and $2-2n$ for $\gamma_1$ puncture. Then for $n\ge 3$, we can assume the configuration is empty. 
\end{proof}

\begin{proposition}\label{prop:vanish}
	If $W$ is an exact filling of $(\RP^{2n-1},\xi_{std})$ for $n\ge 3$, then $SH^{*,\le 2+\epsilon}_+(W;\R) \to H^{*+1}(W;\R)$ is surjective. 
\end{proposition}
\begin{proof}
	In view of Proposition \ref{prop:kill}, it is sufficient to prove that there is a class $1+A\in \oplus_{i\ge 0}H^{2i}(W;\R)$ for $A\in \oplus_{2i>0} H^{2i}(W;\R)$ is mapped to zero in $SH^{*,\le 2+\epsilon}(W;\R)$. Since we have a $\Z_2$ grading, it is sufficient to show that the composition $SH^{*,\le 2+\epsilon}_+(W;\R) \to H^{*+1}(W;\R) \stackrel{\text{projection}}{\longrightarrow} H^0(W;\R) =\R$ is nonzero by the tautological long exact sequence \eqref{eqn:exact}. 
	
	We consider the generator $\check{\gamma}_0^2$, it is not necessarily a closed class in the positive cochain $(C_+,d_+)$. However, we claim that $d_+(\check{\gamma}_{0}^2)$ can only have nonzero components in $\hat{\gamma}_0$ for a sufficiently stretched $J_t$. Again by contact action and homology reason, the only possible configuration after the neck-stretching is with negative end asymptotic to either $\check{\gamma}_0$ or $\hat{\gamma}_0$  and one negative puncture asymptotic to $\gamma_0$. Since $\mu_{CZ}(\check{\gamma}_0)$ and $\mu_{CZ}(\check{\gamma}^2_0)$ has the same parity, then we have the only contribution is to $\hat{\gamma_0}$.\footnote{It indeed contributes to the differential in any case there are rigid holomorphic plane bounded by $\gamma_0$ in $W$, see Remark \ref{rmk:n=2}}.  By Proposition \ref{prop:MB}, we have that $d_+(\check{\gamma}_1)=a_1\hat{\gamma}_0$ for $a_1\ne 0$.  If we write $d_+(\check{\gamma}_0^2)=k\hat{\gamma}_0$, then $\check{\gamma}^2_0-\frac{k}{a_1}\check{\gamma}_1$ is closed in the positive cochain complex. Then by Proposition \ref{prop:unit} and Proposition \ref{prop:other}, we have $SH^{*,\le 2+\epsilon}(W;\R) \to H^{*+1}(W;\R) \stackrel{\text{projection}}{\longrightarrow} H^0(W;\R) =\R$ is nonzero.  
 \end{proof}

\begin{remark}\label{rmk:n=2}
	The reason that our proof does not work for $n=2$ is because Proposition \ref{prop:other} does not hold for $n=2$. Indeed, for the fully stretched almost complex structure, the algebraic count of the top curve is $1$. Hence the contribution $\cM_{\check{\gamma}_1,q}$ is reduced to the augmentation to $\gamma_0$. Then we can discuss the following two cases,
	\begin{enumerate}
		\item When $W$ is the exact filling $T^*S^2$, then the augmentation is $2$. One can see it from the completion of $T^*S^2$ into $S^2\times S^2$. Moreover, one can show that $\rd_+(\check{\gamma}^2_0)=2\hat{\gamma}_0$ by the neck-stretching argument and augmentation. Then using $d_+(\hat{\gamma}_1)=2\check{\gamma}_0$ (c.f. discussion before Proposition \ref{prop:MB}), one sees that $1$ is not killed at least in $SH^{\le 2+\epsilon}(W;\R)$. The full computation in this spirit was carried out in \cite{diogo2019symplectic}.
		\item When $W$ is the strong filling $\cO(-2)$, then the augmentation is $t^{-1}$, where $t$ is the formal variable of degree $0$ to keep track of the intersection with divisor $\CP^1$ in the Novikov field. Then $\rd_+(\check{\gamma}^2_0)=t^{-1}\hat{\gamma}_0$. As a consequence $\check{\gamma}^2_0-\frac{t^{-1}}{2}\hat{\gamma}_1$ is closed in the positive symplectic cohomology and is mapped to $2-\frac{t^{-2}}{2}$.  Then $SH^*(W;\Lambda)=0$, this coincides with the result in \cite{ritter2014floer}.
	\end{enumerate}
\end{remark}

\begin{remark}\label{rmk:ritter}
	Ritter \cite{ritter2014floer} showed that for $n\ge 3$,  $SH^*(\cO(-2))=\Lambda[\omega]/(\omega^{n-2}-4t)$, where $t$ is the formal variable\footnote{The $t$ in \cite{ritter2014floer} is different from the $t$ in Remark \ref{rmk:n=2}, in the sense that $t$ in this remark is the generator of $H_2(\CP^{n-1})$, which intersects $\CP^{n-1}$ $n-3$ times. Then it is easy to see their equivalence.} in the Novikov field $\Lambda$ and $\omega$ is the generator of $H^2(\CP^{n-1};\R)$.  On the other hand, the quantum cohomology $QH^*(\cO(-2)) = \Lambda[\omega]/(\omega^{n}-4t\omega^2)$. Therefore the positive symplectic cohomology is the quotient $QH^*(\cO(-2))/SH^*(\cO(-2))$, which can be viewed as generated by $\omega^{n-1}-4t\omega$ and $\omega^{n-2}-4t$. In the Morse-Bott spectral sequence, the former is represented by multiples of $\check{\gamma}_0$ and the latter is represented by multiples of $\check{\gamma}^2_0$. Moreover, $\omega^{n-2}-4t$ projected to $H^0(W;\Lambda)$ is indeed a unit in $\Lambda$. However, $\omega^{n-2}-4t$ is a zero divisor in the quantum cohomology, hence it does not lead to the vanishing of symplectic homology. 
\end{remark}

\begin{proof}[Proof of Proposition \ref{prop:key}]
	By Proposition \ref{prop:MB} and \ref{prop:vanish}, we have $\sum \dim H^*(W;\R)\le 2$ and is supported in even degrees. We claim it is only possible for $H^n(W;\R)$ to be nonzero. For otherwise, if we have $H^{2k}(W;\R)=\R$ for $0<2k\ne n$, then $H^{2k}(W)$ contains a $\Z$ summand. Then from the long exact sequence, we know that $H^{2k}(W,\RP^{2n-1})$ also contains $\Z$ summand. Therefore, by Lefschetz duality and universal coefficient theorem, we have $H^{2n-2k}(W)$ also has a $\Z$ summand, which contradicts that the total rank is $\le 2$.
\end{proof}

\section{Generalizations}
In this section, we prove Theorem \ref{thm:gen} using the same argument. The threshold is not optimal. The upshot is for $n >k$, the cohomology of any exact filling of $(S^{2n-1}/\Z_k,\xi_{std})$ will have a bounded free part, which will lead to a contradiction. We first note that $(S^{2n-1}/\Z_k,\xi_{std})$ has a strong filling $\cO(-k)$, i.e.\ the total space of the degree $-k$ line bundle over $\CP^{n-1}$.
\begin{proposition}\label{prop:Z_k}
	Let $W$ be an exact filling of $(S^{2n-1}/\Z_k,\xi_{std})$ for $n>k$, then we have $\sum_{i\in \N}\dim H^{2i}(W;\R) \le k$ and $\sum_{i \in \N} \dim H^{2i+1}(W;\R) \le k-2$. Moreover, $H^{2n-i}(W;\R)=H^i(W;\R)$ for every $0<i<2n$.
\end{proposition}
\begin{proof}
	Similar to the proof of Theorem \ref{thm:main}, we perturb the standard Boothby-Wang contact form to $\alpha_f$ using a $C^2$-small perfect Morse function  $f$ on $\CP^{n-1}$, such that the following holds.
	\begin{enumerate}
		\item Reeb orbits of period smaller than $k+1$ are $\gamma_i^j$ for $0\le i \le n-1, 1 \le j \le k$, where $\gamma_i^j$ is the $j$-multiple cover of $\gamma_i$ and $\gamma_i$ projects to the $i$th critical point $q_i$ of $f$.
		\item The period of $\gamma_j$ is $1+\epsilon_j$.
		\item $\epsilon_j<\frac{\epsilon_{j+1}}{k},\epsilon_j\ll 1$.
	\end{enumerate}
     The third condition will be used to rule out certain terms in $d_+$. In view of these conditions, we can choose the Morse function $f$ to be the following,
     $$\left(\displaystyle\sum_{i=0}^{n-1}\frac{\epsilon_i|z_i|^2}{1+\epsilon_i}\right)\left/ \left(\displaystyle\sum_{i=0}^{n-1}\frac{|z_i|^2}{1+\epsilon_i}\right)\right..
     $$
     Then the pull back of the contact form $\alpha_f$ back to $S^{2n-1}$ is the one given by the ellipsoid $\sum_{i=0}^{n-1} \frac{|z_i|^2}{k(1+\epsilon_i)}=1$. With suitable choice of $\epsilon_i$, we may assume it is a non-degenerate ellipsoid.
     
     Similar to the proof of Proposition \ref{prop:unit}, we separate the proof into several steps.
     
     \textbf{Step 1:} For a generic and sufficiently stretched almost complex structure, $\cM(\check{\gamma}_0^k,q)=\cM_{\check{\gamma}_0^k,q}$. 
     
     Assume there are multi-level cascades, then the negative cylindrical end of the top cascades level must be asymptotic to $\overline{\gamma}_i^j$ for $j<k$ by symplectic action. After fully stretching the almost complex structure, the top cascades level must develop a set of negative punctures asymptotic to $\Gamma=\{\gamma_{i_1}^{j_1},\ldots,\gamma_{i_m}^{j_m}\}$, then we have $\sum_{s=1}^m j_s+j=0\mod k$ by homology reasons. Among all such configurations, the only cases with non-negative contact energy are $i=i_1=\ldots=i_m=0$ and $\sum_{s=1}^m j_s+j=k$. If the next cascades level is not the bottom level, then by the same argument, the negative cylindrical end must be asymptotic to $\gamma_0^{j'}$ for $j'<j$. We can keep the argument going and conclude that for sufficiently stretched almost complex structure, the bottom cascade level must have  a positive cylindrical end asymptotic to $\overline{\gamma}_0^s$ for $s<k$. Then in the fully stretched picture, this bottom level must have negative punctures asymptotic to $\gamma_0^{j_1},\ldots \gamma_0^{j_m}$ with $j_1+\ldots+j_m=s$ by the same argument as before. Note that the Conley-Zehnder index of $\gamma_i^j$ using the bounding disk in $\cO(-k)$ is 
     $$\mu_{CZ}(\gamma_i^j)+n-3=2i+2j-2.$$ 
     Note the bottom curve has the point constraint from $q$. The virtual dimension of this configuration (the positive cylindrical end has no orbit point constraint) is 
     $$\mu_{CZ}(\gamma_0^s)-n-\sum_{i=1}^m (\mu_{CZ}(\gamma_0^{j_i})+n-3)=2s-2n+1-\sum_{i=1}^m (2j_i-2)=2m-2n+1\le 2s-2n+1<0.$$
     Hence for generic and sufficiently stretched $J_t$, we do not have any multi-level cascades contributing to $\la d_{+,0}(\check{\gamma}^k_0), q\ra$. 
     
     \textbf{Step 2: }For a generic and sufficiently stretched almost complex structure, we have that moduli space of $1$-level cascades $\cM(\check{\gamma}_0^k,q)$ is identified with  the fully stretched moduli space of $1$-level cascades $\cM^{\infty}(\check{\gamma}_0^k,q)$ and both are compact. 
     
     Similar to the step 2 of Proposition \ref{prop:unit}, it is enough to prove the compactness of $\cM^{\infty}(\check{\gamma}_0^k,q)$. The multi-level cascades are ruled out by step 1. We only need to rule out the case with negative punctures. Again by action and homology reason, the potential breakings are those with negative punctures $\gamma^{k_1}_0,\ldots,\gamma^{k_j}_0$ for $\sum k_i=k$. But the expected dimension of this moduli space is 
     $$\mu_{CZ}(\gamma_0^k)-n-1-\sum_{s=1}^j(\mu_{CZ}(\gamma^{k_s}_0)+n-3)=2k-2n-\sum_{s=1}^j(2k_s-2)=2j-2n<0.$$
     Hence such moduli space is empty for generic $J_t$.

     \textbf{Step 3:} For a generic almost complex structure, $\cM^{\infty}(\check{\gamma}_0^k,q)=k$. 
     
     Using Proposition \ref{prop:curve} and the fact that $\gamma_0^k$ is a good Reeb orbit, this claim follows from the same argument in step 3 of Proposition \ref{prop:unit}. So far we have proven that $\la d_{+,0}(\check{\gamma}^k_0), q\ra =\#\cM_{\check{\gamma}_0^k,q}=\# \cM(\check{\gamma}_0^k, q)=\# \cM^\infty(\check{\gamma}_0^k, q)=k$ for generic and sufficiently stretched almost complex structures.

     \textbf{Step 4:} For a sufficiently stretched almost complex structure,   we have $d_+(\check{\gamma}^k_0)=\sum_{i=1}^{k-1} b_i \hat{\gamma}^i_0$. 
     
     For this we use the similar neck-stretching argument as in Proposition \ref{prop:vanish}. By parity of generators, we only need to consider  $\la d_+(\check{\gamma}^k_0),\hat{\gamma}_j^i\ra$. In the fully stretched picture, the curve in $\widehat{X}$ (could have multiple cascades levels) with maximal contact energy is the one with $k-i$ negative punctures asymptotic to $\gamma_0$. In this case, we have 
	$$\cC(\gamma_0^k)-\cC(\gamma_j^i)-(k-i)\cC(\gamma_0)=k(1+\epsilon_0)-i(1+\epsilon_j)-(k-i)(1+\epsilon_0)=i(\epsilon_0-\epsilon_j)< 0, \text{ if } j\ne 0.$$
	As a consequence of \eqref{eqn:action}, we have that $d_+(\check{\gamma}^k_0)=\sum_{i=1}^{k-1} b_i \hat{\gamma}^i_0$.
	
	\textbf{Step 5:} For $\epsilon_i$ sufficiently small and sufficiently stretched almost complex structure and $i<k,j>0$, we have 
	\begin{equation}\label{eqn:diff}
	d_+(\check{\gamma}^i_j)=k\hat{\gamma}^i_{j-1}+\sum_{m<i,l\le j} a_{m,l}\hat{\gamma}^{m}_l.
	\end{equation}

	The proof of the first term follows from the same argument in Proposition \ref{prop:MB}. We use a filtered symplectic cohomology with action window around $i$, which is generated by check and hat generators of $\overline{\gamma}^i_0,\ldots \overline{\gamma}^i_{n-1}$. Then we find a filling $W^{std}_{1-\delta}\subset W$ of the standard Morse-Bott contact form $(1-\delta)\alpha_{std}$. Using the functorial argument in Proposition \ref{prop:MB} and Proposition \ref{prop:Viterbo'}, we can conclude that the filtered symplectic cohomology is isomorphic to $H^*(S^{2n-1}/\Z_k;\R)$ via the Viterbo transfer map, where $S^{2n-1}/\Z_k$ is the critical submanifold in the free loop space parameterizing the space of parameterized Reeb orbits of multiplicity $i$ of $\alpha_{std}$. For $\epsilon,\delta$ sufficiently small, the isomorphism respects a local $\Z$-grading, this implies that  
	$$d_+(\check{\gamma}^i_j)=c^i_j\hat{\gamma}^i_{j-1}+\text{ orbits with lower multiplicty.}$$
	for $c^i_j\ne 0$. When using $\Z$-coefficient, we can get $c^i_j=\pm k$. A further comparison on orientations can conclude that $c^i_j$ can be chosen to be $k$ with suitable orientation data. For the argument below, using $c^i_j\ne 0$ is sufficient, but for the simplicity of notation we use $c^i_j=k$.
    
    For the contribution from orbits with lower multiplicity, i.e.\ $\la d_+(\check{\gamma}^i_j), \hat{\gamma}_l^m\ra$ with $m<i$, the curve with maximal contact action difference in the fully stretched picture is those with $i-m$ negative punctures asymptotic to $\gamma_0$. Then 
	\begin{eqnarray*}
	\cC(\gamma_j^i)-\cC(\gamma_l^m)-(i-m)\cC(\gamma_0) &= &i(1+\epsilon_j)-m(1+\epsilon_l)-(i-m)(1+\epsilon_0) \\
	&= &i\epsilon_j -m\epsilon_l-(i-m)\epsilon_0 \\
	&< & \frac{i}{k}\epsilon_l-m\epsilon_l-(i-m)\epsilon_0<0, \text{ if } l>j.
	\end{eqnarray*}
	Therefore the claim follows from \eqref{eqn:action}.

	\textbf{Step 6}: We claim that
	\begin{equation}\label{eqn:lazy}
	d_+(\check{\gamma}^k_0+\sum_{i=1}^{k-1}\sum_{j=1}^{k-i} c_{i,j} \check{\gamma}^i_j)=0,
	\end{equation}
	where $c_{i,j}\in \R$ is defined recursively by $a_{m,l}$.

	To see \eqref{eqn:lazy}, first note that by \eqref{eqn:diff}, we have
	$$d_+(\check{\gamma}_0^k-\sum_{i=1}^{k-1}\frac{b_i}{k}\check{\gamma}_1^i)=\sum_{i=1}^{k-2} \sum_{j=0}^{1} d_{i,j} \hat{\gamma}^i_j,$$
	Then we have 
	$$d_+(\check{\gamma}_0^k-\sum_{i=1}^{k-1}\frac{b_i}{k}\check{\gamma}_1^i-\sum_{i=1}^{k-2}\sum_{j=0}^1\frac{d_{i,j}}{k}\check{\gamma}^i_{j+1})=\sum_{i=1}^{k-3}\sum_{j=0}^{2} e_{i,j}\hat{\gamma}^i_j.$$
	Then we can keep applying the argument to obtain \eqref{eqn:lazy}.
	
	\textbf{Step 7:} For a generic and sufficiently stretched almost complex structure, we have $\la d_{+,0}(\check{\gamma}^i_j), q \ra=0$ for $i+j\le k,j\ge 1$. 
	
	We first claim that there is no multi-level cascades contributing to $\la d_{+,0}(\check{\gamma}^i_j), q \ra$. Assume otherwise, the top cascades level's negative cylindrical end is asymptotic to $\overline{\gamma}^m_l$ for $m\le i$. In the fully stretched picture, the curve with maximal contact energy are those with $i-m$ (which could be zero, when $i=m$) negative punctures asymptotic to $\gamma_0$. The contact energy is given by 
	 $$\cC(\gamma_j^i)-\cC(\gamma^m_l)-(i-m)\cC(\gamma_0)=i\epsilon_j-m\epsilon_l-(i-m)\epsilon_0<0, \text{ if } l>j.$$
	 Therefore we must have $l\le j$. Of course, the top cascades level's negative cylindrical end cannot be asymptotic to $\overline{\gamma}^i_j$, as this would force the curve to be a trivial cylinder. We can keep the argument going and conclude the bottom level of the cascades must have the positive cylindrical end asymptotic to $\overline{\gamma}^m_l$ with $m\le i,l\le j$ and the equality does not holds simultaneously. In particular, $m+l<i+j$. We consider this bottom cascades level in the fully stretched picture, since $m<i+j\le k$, we must have negative punctures by homology reason. Then the maximal virtual dimension of the curve in $\widehat{X}$ is from the curve with $m$ negative punctures asymptotic to $\gamma_0$, which is 
	 $$\mu_{CZ}(\gamma_l^m)-n-m(\mu_{CZ}(\gamma_0)+n-3)=2l+2m+1-2n<2i+2j-2n\le 2k-2n<0.$$
	 In particular, there is no multi-level contribution. Next we will rule out the single level cascades. In the fully stretched situation, since $j\ge 1$, we have $i<i+j\le k$ and the curve must have negative punctures in the topmost SFT level by homology reason. The maximal virtual dimension of the topmost level is from the curve with $i$ negative punctures asymptotic to $\gamma_0$. Hence the virtual dimension of the top level curve is at most
	 $$\mu_{CZ}(\gamma^i_j)-n-1-i(\mu_{CZ}(\gamma_0)+n-3)=2i+2j-2n\le 2k-2n<0.$$
	 Assembling the results above, we know that the closed cochain $\check{\gamma}^k_0+\sum_{i=1}^{k-1}\sum_{j=1}^{k-i} c_{i,j} \check{\gamma}^i_j$ is mapped to $k$ in $H^0(W;\R)$. As a result, we have $SH_+^{*,\le k+\epsilon_1}(W;\R)\to H^{*+1}(W;\R)$ is surjective by Proposition \ref{prop:kill}.
	 
	 \textbf{Step 8:} For $\epsilon_i$ sufficiently small $\Ima(SH_+^{*,\le k+\epsilon_1}(W;\R)\to H^{*+1}(W;\R))$ has rank at most $2k-2$.
	 
	 By the same argument of Proposition \ref{prop:MB}, $SH_+^{*,\le k+\epsilon_1}(W;\R)$ can be assembled from $k$ filtered symplectic cohomology with action window around $1,\ldots,k$ by iterating the tautological long exact sequences. More precisely, the cochain complex of the first $k-1$ filtered symplectic cohomology is generated by $\check{\gamma}_0^i,\hat{\gamma}_0^i,\ldots,\check{\gamma}_{n-1}^i,\hat{\gamma}_{n-1}^i$ for $1\le i \le k-1$, and the cohomology is $H^*(S^{2n-1}/\Z_k; \R)$ and generated by $\check{\gamma}_0^i,\hat{\gamma}_{n-1}^i$.  The cochain complex of the last filtered symplectic cohomology is generated by $\check{\gamma}^k_0$ and $\hat{\gamma}^k_0$, which also generate the cohomology. By the same argument in Proposition \ref{prop:MB}, $\hat{\gamma}^k_0$ will be killed when we increase the action threshold and  $\hat{\gamma}_{n-1}$ does not map to nontrivial class in $H^*(W;\R)$ by $S^1$ symmetry which is guaranteed by the $S^1$-equivariant transversality. Therefore $\Ima(SH_+^{*,\le k+\epsilon_1}(W;\R)\to H^{*+1}(W;\R))$ has rank at most $2k-2$, of which at most $k$ are from check orbits and at most $k-2$ are from hat orbits. 
	 
	 Then the proposition follows from that check orbits have odd grading and hat orbits have even grading. The last part is a consequence of Lefschetz duality and universal coefficient theorem.
\end{proof}

\begin{proposition}\label{prop:Z_p}
	Let $p$ be an odd prime, then for any strong filling $W$ of $(S^{2n-1}/\Z_p, \xi_{std})$, we have $H^{2i}(W)\to H^{2i}(S^{2n-1}/\Z_p)=\Z_p$  is surjective if $0<i<n$ and in $p$-adic representation, $i$ is digit-wise no larger than $n$, i.e.\ if we write $n=\sum_{s=0}^{\infty}a_s p^s, i=\sum_{s=0}^{\infty} b_sp^s$ then $b_s\le a_s$ for all $s\ge 0$.
\end{proposition}
\begin{proof}
	The proof is similar to Proposition \ref{prop:Chern}. We first compute the Chern classes of $\xi_{std}$ using the standard filling $\cO(-p)$, i.e.\ the total space of degree $-p$ line bundle over $\CP^{n-1}$. The total Chern class of the total space $\cO(-p)$ is $(1+u)^n(1-pu)$, where $u$ is generator of $H^2(\cO(-p))=H^2(\CP^{n-1})$. We write $n$ in $p$-adic as $\sum_{s=0}^k a_s p^s$. Then using the fact $(\sum x_*)^p=\sum x_*^p \mod p$ and $\binom{m}{l} \ne 0 \mod p$ whenever $0\le l\le m < p$. We have the following
	\begin{equation}\label{eqn:padic}
	(1+u)^n(1-pu)=(1+u)^n=\prod_{s=0}^k (1+u^{p^s})^{a_s} = \prod_{s=0}^k(1+\sum_{j=1}^{a_s}c_{s,j}u^{p^sj}) \mod p,
	\end{equation}
	for $c_{s,j}=\binom{a_s}{j}\ne 0 \mod p$. That is in \eqref{eqn:padic}, the monomials with non-constant coefficient are those $u^{\sum_{s=0}^kp^sj_s}$ for $j_s\le a_s$, i.e.\ the degree is digit-wise smaller or equal to $n$ in $p$-adic representation. In other words, the $i$th Chern class of the total space $\cO(-p)$ mod $p$ is nonzero iff in $p$-adic representation, $i$ is digit-wise no larger than $n$. By the Gysin sequence, we know $H^{2i}(\cO(-p))\to H^{2i}(S^{2n-1}/\Z_p)=\Z_p$ is the mod $p$ map for $0<i<n$ and $c_i(\cO(-p))|_{S^{2n-1}/\Z_p}=c_i(\xi_{std})$. Therefore we have $c_i(\xi_{std})\ne 0$ for any such $i$. Since for any strong filling $W$, we have $c_i(W)|_{S^{2n-1}/\Z_p}=c_i(\xi_{std})$, the claim follows.
\end{proof}
\begin{proof}[Proof of Theorem \ref{thm:gen}]
	Let $I$ be the set of $i$ such that $0<i<n$ and is digit-wise no larger than $n$ in $p$-adic representation. A basic observation is that if $i\in I$ then $n-i\in I$. It is clear that  $|I| = \sum a_s - 2$. Then $I_{\frac{1}{2}}:=I \cap (0,\frac{n}{2})$ has at least $\frac{1}{2}\sum a_s -2$ element.  Note that if $\sum a_s > 3p-3$, then $n>p$. Then we can apply Proposition \ref{prop:Z_k} to get that $\sum_{0\le 2j < n}\dim H^{2j}(W;\R)$ is at most $\lfloor \frac{p}{2}\rfloor=\frac{p-1}{2}$ and $\sum_{0\le 2j+1 < n}\dim H^{2j+1}(W;\R)$ is at most $\lfloor \frac{p-2}{2}\rfloor =\frac{p-3}{2}$ as $p$ is odd. We want to find $i\in I_{\frac{1}{2}}$, such that $H^{2i-1}(W),H^{2i}(W),H^{2i+1}(W)$ are torsions.  If $H^{2j}(W;\R)\ne 0$, we can not choose $i=j$, and if $H^{2j+1}(W;\R)\ne 0$, we can not choose $i=j$ or $i=j+1$. Therefore if $\frac{1}{2}\sum a_s-2 > p-3+\frac{p-1}{2}$, i.e.\ $\sum a_s> 3p-3$, we have $|I_{\frac{1}{2}}|>\sum_{0\le 2j < n}\dim H^{2j}(W;\R)+2\sum_{0\le 2j+1 < n}\dim H^{2j+1}(W;\R)$. In particular, we have $i\in I_{\frac{1}{2}}$, such that $H^{2i-1}(W),H^{2i}(W), H^{2i+1}(W)$ are torsions. Then by symmetry, we also have $H^{2n-2i-1}(W), H^{2n-2i}(W), H^{2n-2i+1}(W)$ are all torsions. Then Lefschetz duality and universal coefficient theorem imply that 
	$$0\to H^{2n-2i+1}(W) \to H^{2i}(W) \to \Z_p \to H^{2n-2i}(W) \to H^{2i+1}(W) \to 0,$$
	$$0\to H^{2i+1}(W) \to H^{2n-2i}(W) \to \Z_p \to H^{2i}(W) \to H^{2n-2i+1}(W) \to 0,$$
	But since $i, n-i \in I$, we have $H^{2i}(W) \to \Z_p$ and $H^{2n-2i}(W) \to \Z_p$ above are both surjective by Proposition \ref{prop:Z_p}. Then $H^{2n-2i}(W) \to H^{2i+1}(W)$ and $H^{2i}(W) \to H^{2n-2i+1}(W)$ must be isomorphisms. Since all of the group are torsions, from the long exact sequences we have both $|H^{2i}(W)|=|H^{2n-2i+1}(W)|$ and $|H^{2i}(W)|=p|H^{2n-2i+1}(W)|$, which is a contradiction.
\end{proof}

\begin{remark}\label{rmk:final}
	If one uses the polyfold technique in \cite{zhou2018quotient} to achieve $S^1$-equivariant transversality. We can bring the rank of $H^*(W;\R)$ down to $k$, since those hat orbits will not contribute to $H^*(W;\R)$ as in the proof of Proposition \ref{prop:key}. Observe that the check orbit will be mapped to even degree cohomology of $W$. Then we can improve Theorem \ref{thm:gen} by a factor to $\sum a_s>p+3$. It is interesting to note that $n\ge k+1$ in Proposition \ref{prop:Z_k} is the threshold for $\C^n/\Z_k$ to be a terminal singularity. By \cite{mclean2016reeb}, the terminality of a singularity is equivalent to that the link has a contact form with positive rational SFT degrees.
	
	The key observation in this paper is that $SH^*_+(W;\R)\to H^{*+1}(\partial W;\R)$ contains $1$ in the image for the hypothetical exact filling $W$, which bears a lot similarity with results in \cite{zhou2019symplectic}. In view of Ritter results \cite{ritter2014floer}, $SH^*_+(\cO(k);\Lambda)\to H^{*+1}(S^{2n-1}/\Z_k;\Lambda)$ is very likely to be isomorphic to the hypothetical $SH^*_+(W;\Lambda)\to H^{*+1}(S^{2n-1}/\Z_k;\Lambda)$ for $n\ge k+1$.  The invariance phenomenon here has gone beyond those in \cite{zhou2019symplectic} as we have multiple augmentations.  On the other hand, when $n\le k$, as we seen from the $n=k=2$ case, the map from positive symplectic cohomology to the cohomology of boundary depends on the filling. For higher $n\le k$ examples, although we do not know if there are more fillings, but there are algebraic augmentations which would change whether $1$ is in the image of the map from linearized non-equivariant contact homology to the cohomology of the boundary. The $n=k$ case is indeed the limit of our method, as our symplectic part does not differentiate exact fillings with Calabi-Yau fillings (see Remark \ref{rmk:other}) and $\C^n/\Z_n$ indeed carries a Calabi-Yau filling with the right rank of cohomology.
\end{remark}
In case of strong fillings, the sequence \eqref{eqn:exact} still holds after we replace $H^*(W;\R)$ by the quantum cohomology $QH^*(W;\Lambda)$. As a group, $QH^*(W;\Lambda)\simeq H^*(W;\Lambda)$, but the map $QH^*(W;\Lambda) \to SH^*(W;\Lambda)$ is a unital ring map if we use the deformed quantum ring structure on $QH^*(W;\Lambda)$. 

In view of the proof of Proposition \ref{prop:kill}, to imply the vanishing of symplectic cohomology and then a contradiction, we need to show that $1+A$ is never a zero divisor. Hence we have the following. 
\begin{corollary}\label{cor:cor}
	Let $W$ be any (semi-positive \cite{mcduff2012j}) strong filling of the contact manifolds in Theorem \ref{thm:main} and \ref{thm:gen}, then the quantum cohomology $QH^*(W;\Lambda)$ has a zero divisor in the form of $1+A$ for $A\in\oplus_{i>0} H^{2i}(W;\Lambda)$.
\end{corollary}	
\begin{remark}
	Symplectic cohomology for general strong fillings requires virtual techniques to deal with sphere bubbles in general. If one wishes to avoid the technical overhead, $W$ should be restricted to the case of semi-positive strong fillings. Then the symplectic cohomology can be defined as usual. The filtered positive symplectic cohomology $SH_+^{*,\le a}(W;\Lambda)$ can also be defined, based on the asymptotic behavior lemma \cite[Lemma 2.3]{cieliebak2018symplectic} instead of the action filtration\footnote{The filtered symplectic cohomology is actually filtered by the contact action, which roughly coincides with the negative symplectic action when the filling is exact.}. See \cite[\S 8]{zhou2019symplectic} for a setup for the Calabi-Yau case, the general semi-positive case is similar. 
\end{remark}
\begin{proof}[Proof of Corollary \ref{cor:cor}]
	By the argument for Theorem \ref{thm:main} and \ref{thm:gen}, we have $SH_+^{*,\le a}(W;\Lambda)\to QH^{*+1}(W;\Lambda)\stackrel{\text{projection}}{\longrightarrow} H^0(W;\Lambda)$\footnote{We use $H^0(W;\Lambda)$ here to stand for the $\Lambda$-space spanned by $1\in H^0(W;\Lambda)$. Note that the degree $0$ part $QH^0(W;\Lambda)$ may be different from $H^0(W;\Lambda)$.} hits $1$ for a suitable $a$ for any (semi-positive) strong filling $W$. Note that for any $x\in QH^{*}(W;\Lambda)$ the quantum product $x\cup \cdot: QH^*(W;\Lambda)\to QH^{*+|x|}(W;\Lambda)$ is $\Lambda$-linear map between finite dimensional $\Lambda$-spaces. Therefore $x$ is either an invertible element or a zero divisor. If there is no zero divisor in the form of $1+A$ for $A\in\oplus_{i>0} H^{2i}(W;\Lambda)$, we know that $SH_+^{*,\le a}(W;\Lambda)\to QH^{*+1}(W;\Lambda)$ hits an invertible element. Then the proof of Proposition \ref{prop:kill} goes through and $SH_+^{*,\le a}(W;\Lambda)\to QH^{*+1}(W;\Lambda)$ is surjective. Then we can derive a topological contradiction as before.
 \end{proof}
\begin{remark}\label{rmk:other}
	Note that the exactness is used to get $H^*(W;\R)\to SH^*(W;\R)$ is a unital ring map. Then we use the fact that  $1+A$ is a unit in $H^*(W;\R)$ for $A\in \oplus_{2i>0}H^{2i}(W;\R)$, as $A$ is a nilpotent element. This property also holds for symplectically aspherical filling or fillings with undeformed quantum cohomology, as showed by Corollary \ref{cor:cor}. On the other hand, if the filling is Calabi-Yau, i.e.\ $c_1(W)=0$ in $H^2(W;\Q)$, then we have a $\Z$ grading and $A$ is necessarily $0$. Although the multiplicative structure might be deformed, $1$ is always a unit in $QH^*(W;\Lambda)$, hence we have surjectivity of $SH^{*,\le a}_+(W;\Lambda)\to QH^{*+1}(W;\Lambda)$ for a suitable $a$. In other words, our proof shows that contact manifolds in Theorem \ref{thm:main} and \ref{thm:gen} do not have symplectically aspherical or Calabi-Yau fillings. 
\end{remark}
Combining with $\Z_2$ and $\Z_3$ quotient singularities, we can prove Theorem \ref{thm:strong}.
\begin{proof}[Proof of Theorem \ref{thm:strong}]
	In view of Theorem \ref{thm:main}, we only need to prove the case for $n=2^k\ge 4$. In this case, we will use $(S^{2n-1}/\Z_3,\xi_{std})$. Let $W$ be an exact filling of $(S^{2n-1}/\Z_3,\xi_{std})$. By Proposition \ref{prop:Z_k}, we know that $1\le \sum_{i\in \N}\dim H^{2i}(W;\R) \le 3$ and $\sum_{i\in \N} \dim H^{2i+1}(W;\R)=0$. If $H^2(W;\R) =0$, then we know $H^1(W),H^2(W),H^3(W), H^{2n-1}(W),H^{2n-2}(W),H^{2n-3}(W)$ are all torsions. Moreover by Proposition \ref{prop:Z_p}, we have $H^2(W)\to H^2(S^{2n-1}/\Z_3)$ and $H^{2n-2}(W)\to H^{2n-2}(S^{2n-1}/\Z_3)$ are surjective. Then we arrive at a contradiction by the same argument in Theorem \ref{thm:main} and \ref{thm:gen}.
	
	In the case of $\dim H^2(W;\R)\ge 1$, we must have $H^{n}(W;\R) = 0$ with $n=2^k$ even. Then we have the following long exact sequence
	$$0\to H^{n}(W, S^{2n-1}/\Z_3) \to H^{n}(W) \to \Z_3 \to H^{n+1}(W,S^{2n-1}/\Z_3) \to H^{n+1}(W) \to 0.$$
	Since $H^{n-1}(W),H^n(W),H^{n+1}(W)$ are all torsions, then Lefschetz duality and universal coefficient theorem imply that $H^{n}(W, S^{2n-1}/\Z_3) \simeq H^{n+1}(W)$ and $H^{n+1}(W,S^{2n-1}/\Z_3) \simeq H^n(W)$.  The long exact sequence then becomes
	$$0\to H^{n+1}(W) \to H^{n}(W) \to \Z_3 \to H^{n}(W) \to H^{n+1}(W) \to 0.$$
	which will contradict that they are all torsions. 
\end{proof}

Inspired by Remark \ref{rmk:final}, we end this paper with the following conjecture.
\begin{conjecture}
	If the isolated quotient singularity $\C^n/G$ for finite nontrivial $G\in U(n)$ is terminal, then the link of the  singularity does not have symplectically aspherical fillings or Calabi-Yau fillings.
\end{conjecture}

\appendix
\section{}
In the following, we prove the property for Viterbo transfer map used in Proposition \ref{prop:MB}. We first recall an alternative way of defining the filtered positive symplectic cohomology $SH^{*,\le a}_+(W)$ following \cite{cieliebak2018symplectic}\footnote{Strictly speaking, \cite{cieliebak1996applications} uses the Hamiltonians that is $C^2$ small Morse on $W$, the equivalence of these two models can be obtained by an argument similar to \cite[Proposition 2.10]{zhou2019symplectic}.}. Let $\cH$ denote the set of admissible Hamiltonians with slope that is not the period of a Reeb orbit on $\partial W$. For $H\in \cH$, we define can define  $C^{\le a}(H)$ to be the subcomplex generated by critical point of $g$ (i.e.\ those with zero symplectic action) and $\check{\gamma},\hat{\gamma}$ with $\cA_{H}(\overline{\gamma})\ge -a$. We define $C^{\le a}_+(H)$, or equivalently, $C^{(0,a]}(H)$, to be the quotient complex of $C^{\le a}(H)$ generated only by $\check{\gamma}$ and $\hat{\gamma}$.\footnote{Note that we have a sign difference in the convention of symplectic action compared to \cite{cieliebak1996applications}.} For $H\le K$ in $\cH$, the continuation map induces maps $f^{\le a}_{HK}:C^{\le a}(H)\to C^{\le a}(K)$ and $f^{\le a}_{+,HK}C_+^{\le a}(H)\to C_+^{\le a}(K)$, both satisfy the obvious functorial property\footnote{Our $f_{HK}$ is the $f_{KH}$ in \cite{cieliebak1996applications}. We choose this convention, so that everything is parsed from left to right, i.e.\ $f_{HK}$ is from $H$ to $K$. Similarly, $\cM_{a,b}$ is this paper counts differential from $a$ to $b$. }. Then we define
$$SH^{\le a}(W):=\varinjlim_{H\in \cH}H^*(C^{\le a}(H)), \quad SH^{\le a}(W)_+:=\varinjlim_{H\in \cH}H^*(C_+^{\le a}(H))$$
\begin{proposition}\label{prop:iso2}
	The above definition is equivalent to the definition in \S \ref{s2}.
\end{proposition}
\begin{proof}
	We consider a special class of Hamiltonians $\cH'$, such that $H\in \cH'$ iff $H$ on $\partial W \times (1,1+w)$ coincides with a Hamiltonian in $\cH_a$ with width $\frac{w}{2}$ and $H\in \cH_b$ for some $b>a$. Then $\cH'$ is cofinal in $\cH$. Let $H\in \cH'$, we define $H_a$ to the Hamiltonian equals to $H$ on $W\cup \partial W\times (1,1+w)$ and then extends linearly with slope $a$ to $\widehat{W}$. Note that $C^{\le a}_+(H)=C^{\le a}(H_a)$, as $\R$-module they are same since all other orbits of $H$ have symplectic action $<-a$. Then by the integrated maximal principle, we have the differentials for $C^{\le a}_+(H)$ and $C^{\le a}(H_a)$ can be identified for suitable choice of almost complex structures. For $w$ sufficiently small, we have $C^{\le a}(H_a)=C(H_a)$. Moreover for $H\le K\in \cH'$, we have $H_a\le K_a\in \cH_a$. The functorial property of continuation maps implies that the compositions $C^{\le a}(H_a)\to C^{\le a}(K_a)\to C^{\le a }(K)$ and $C^{\le a}(H_a)\to C^{\le a}(H)\to C^{\le a}(K)$ are homotopic to each other. Therefore we have
	$$\varinjlim_{H\in \cH}H^*(C^{\le a}(H))=\varinjlim_{H\in \cH'}H^*(C^{\le a}(H))=\varinjlim_{H\in \cH'}H^*(C(H_a))=\varinjlim_{H\in \cH_a}H^*(C(H)),$$
	where the last isomorphism is by Proposition \ref{prop:iso}.  The proof for $SH_+^{\le a}(W)$ is identical.
\end{proof}

Let $V\subset W$ be an exact subdomain, then we can define $\cH^{W}(V)$ to be set of Hamiltonians that is $0$ on $V$ and is linear on $\partial W \times (1,\infty)_r$ for $r$ big with slope not a period of Reeb orbits on $\partial W$.  Then such Hamiltonians can be used to compute $SH^*(V)$ by the following.
\begin{lemma}[{\cite[Lemma 5.1]{cieliebak2018symplectic}}]
	For any positive real number $a$ that is not a period of a Reeb orbit on $\partial V$, we have
	$$SH^{*,\le a}(V)=\varinjlim_{H\in \cH^{W}(V)} H^*(C^{\le a}(H)),\quad SH_+^{*,\le a}(V)=\varinjlim_{H\in \cH^{W}(V)} H^*(C^{(0,a])}(H)).$$
\end{lemma}
Then the Viterbo transfer map for filtered symplectic cohomology is defined as follows \cite[Definition 5.2]{cieliebak2018symplectic},
$$\iota_{W,V}^{\le a}:SH^{\le a}(W)\to SH^{\le a}(V), \quad \iota_{W,V}^{\le a}:=\varinjlim_{\substack{H\le k,\\ H\in \cH(W),K\in \cH^W(V)}}f^{\le a}_{HK},$$
$$\iota_{+,W,V}^{\le a}:SH^{\le a}_+(W)\to SH^{\le a}_+(V), \quad \iota_{+,W,V}^{\le a}:=\varinjlim_{\substack{H\le k,\\ H\in \cH(W),K\in \cH^W(V)}}f^{\le a}_{+,HK},$$
where $f^{\le a}_{HK}:C^{\le a}(H)\to C^{\le a}(K), f^{\le a}_{+,HK}:C^{(0, a]}(H)\to C^{(0,a]}(K)$ are the continuation maps.

 The following proposition shows that the Viterbo transfer map is an isomorphism for trivial cobordism if there is no difference in the set of Reeb orbits that generate the filtered cochain complexes.
\begin{proposition}\label{prop:Viterbo}
	Let $W$ be an exact domain, such that the contact form on $\partial W$ is Morse-Bott. For $\delta>0$ we use $W_{1+\delta}$ to denote $W\cup \partial W\times (1,1+\delta]$. Assume $\partial W$ has no Reeb orbit with period in $[\frac{a}{1+\delta}, a]$. Then the Viterbo transfer maps $SH^{*,\le a}(W_{1+\delta};\R)\to SH^{*,\le a}(W;\R)$ is an isomorphism.
\end{proposition}
\begin{proof}
	Let $H\in \cH_a(W_{1+\delta})$. Note that when we view $H$ as a function on $\widehat{W}$, i.e.\ using the $r$-coordinate from $W$, we have the slope of $H$ is $\frac{a}{1+\delta}$. We consider $H'\in \cH_\frac{a}{1+\delta}(W)$, which is a shift of $H$. We claim the Viterbo transfer map $SH^{*,\le a}(W_{1+\delta})\to SH^{*,\le a}(W)$ can be computed by the continuation map $f_{HH'}$. We first take $H''\in \cH_a(W)$ such that $H''\ge H'$, since there is no Reeb orbits of $\partial W$ with period in $[\frac{a}{1+\delta},a]$. The combination of arguments in Proposition \ref{prop:iso} and \ref{prop:iso2} yield that $f_{H'H''}$ is an quasi-isomorphism. Hence it is sufficient to prove $f_{HH''}$ computes the Viterbo map. By Proposition \ref{prop:iso2}, the Viterbo transfer map can be compute by 
	$$\varinjlim_{\substack{H\le K,\\ K\in \cH^{W_{1+\delta}}(W)}} f^{\le a}_{HK}.$$
	By the argument in \cite[Lemma 5.1]{cieliebak2018symplectic}, we can find a cofinal family of functions $K\in \cH^W(W)$, such that all generators with symplectic action $>-a$ are contained in an arbitrarily small neighborhood of $W$ where $K$ behave like a function in $\cH_b(W)$ form some $b\gg a$. We use $K_b$ to denote the truncation as in the proof of Proposition \ref{prop:iso2}. We can choose the cofinal family has the property that $K_b \ge K$, see the figure below. The integrated maximal principle implies that $f_{KK_b}^{\le a}$ is a quasi-isomorphism. Since $f_{HK_b}^{\le a}=f^{\le a}_{KK_b}\circ f_{HK}^{\le a}$, the Viterbo transfer map can be computed by 
	$$\varinjlim_{\substack{H\le K_b, b\gg a\\ K_b\in \cH_b(W)}} f^{\le a}_{HK_b}.
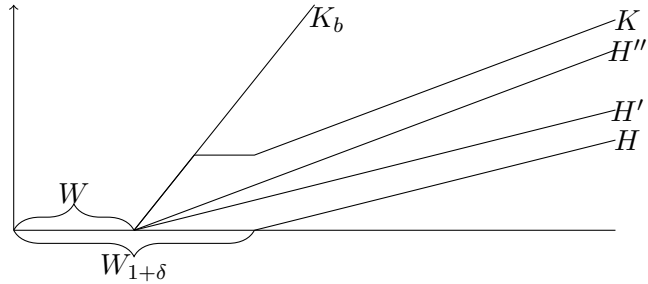
$$
	\begin{figure}[H]
		\begin{center}
			\begin{tikzpicture}[xscale=0.8,yscale=0.2]
			\draw (0,0) to (10,0);
			\draw[->] (0,0) to (0,15);
			\draw (4,0) to (10,6);
			\draw (2,0) to (10,8);
			\draw (2,0) to (10,12);
			\draw (2,0) to (3,5) to (4,5) to (10,14);
			\draw (2,0) to (5,15);
			\node at (10.2,6) {$H$};
			\node at (10.2,8) {$H'$};
			\node at (10.2,12) {$H''$};
			\node at (10.2,14) {$K$};
			\node at (5.2, 14) {$K_b$};
			\draw [decorate,decoration={brace,amplitude=10pt},yshift=-0.2pt](0,0) -- (2,0) node [midway, yshift=1.3em] {$W$};
			\draw [decorate,decoration={brace,amplitude=10pt},yshift=+0.2pt](4,0) -- (0,0) node [midway, yshift=-1.3em] {$W_{1+\delta}$};
			\end{tikzpicture}
		\end{center}
	\end{figure}	
	Again by the argument in Proposition \ref{prop:iso2}, we have $f^{\le a}_{H''K_b}$ is an quasi-isomorphism for $K_b\in \cH_b(W)\ge H''$ with $b\gg a$. Therefore the Viterbo transfer is computed by $f_{HH''}$, hence also $f_{HH'}$. Then we can use the shift trick in Proposition \ref{prop:iso} to show $f_{HH'}$ is a quasi-isomorphism, which concludes the proof.
\end{proof}
As an application of the above proposition, we have the following.
\begin{proposition}\label{prop:Viterbo'}
	Assume $\partial W$ has no Reeb orbit with period in $[\frac{a}{1+\delta}, a]$, then $SH_+^{*,\le a}(W_{1+\delta})\to SH_+^{*,\le a}(W)$ is an isomorphism. For $0<b<a$, if $\partial W$ also has no  Reeb orbit with period in $[\frac{b}{1+\delta}, b]$, then $SH_+^{*,(b,a]}(W_{1+\delta})\to SH_+^{*,(b,a]}(W)$ is an isomorphism.
\end{proposition}
\begin{proof}
	Both claims follow from the previous proposition, the five lemma, and the following commutative long exact sequences.
	$$
	\xymatrix{
	\ldots \ar[r] & SH^{*,\le b}(W_{1+\delta})\ar[r]\ar[d] & SH^{*,\le a}(W_{1+\delta}) \ar[r]\ar[d] & SH^{*,(b,a]}(W_{1+\delta})\ar[r]\ar[d] & \ldots \\
	\ldots \ar[r] & SH^{*,\le b}(W)\ar[r] & SH^{*,\le a}(W) \ar[r] & SH^{*,(b,a]}(W)\ar[r] & \ldots 
   }
  $$
\end{proof}

\bibliographystyle{plain} 
\bibliography{ref}

\end{document}